\documentclass[12pt]{amsart}

\usepackage{amsthm,amsfonts, amssymb, amscd}
\usepackage{youngtab}
\usepackage{ytableau}
\usepackage[all]{xy}
\usepackage{subcaption}

\usepackage{euscript}
\usepackage{cite}

\usepackage[colorinlistoftodos]{todonotes}

\usepackage{tikz}
\usetikzlibrary{matrix,arrows, decorations.markings}
\usetikzlibrary{positioning}

\usepackage[normalem]{ulem}

\usepackage{hyperref}
\hypersetup{%
    linktoc=page
}

\let\oldtocsection=\tocsection
\let\oldtocsubsection=\tocsubsection
\renewcommand{\tocsection}[2]{\hspace{0em}\oldtocsection{#1}{#2}}
\renewcommand{\tocsubsection}[2]{\hspace{1em}\oldtocsubsection{#1}{#2}}

\newtheorem{thm}{Theorem}[section]
\newtheorem{lemma}[thm]{Lemma}

\newtheorem{prop}[thm]{Proposition}

\newtheorem{conjecture}[thm]{Conjecture}
\newtheorem{problem}[thm]{Problem}
\theoremstyle{remark}
\newtheorem{rem}[thm]{Remark}
\theoremstyle{remark}
\newtheorem{example}[thm]{Example}
\theoremstyle{definition}

\theoremstyle{definition}

\theoremstyle{definition}

\newtheorem{definition}[thm]{Definition}

\newtheorem{construction}[thm]{Construction}

\numberwithin{equation}{section}

\newcommand{\C}{\mathbb{C}}           
 
\newcommand{\diag}{\operatorname{diag}}

\newcommand{\codim}{\operatorname{codim}}
\newcommand{\Inv}{\operatorname{Inv}}

\newcommand{\fb}{{\mathfrak b}}

\newcommand{\fg}{{\mathfrak g}}

\newcommand{\fl}{{\mathfrak l}}

\newcommand{\fz}{\mathfrak z}

\newcommand{\gm}{\mu}

 \newcommand{\cb}{\mathcal{B}}

 \newcommand{\cm}{\mathcal{M}}

\newcommand{\cz}{\mathcal{Z}}

\renewcommand{\tilde}{\widetilde}

\renewcommand{\bar}[1]{\overline{#1}}

\newcommand{\Hess}{\mathcal{H}\mathrm{ess}}
\newcommand{\tab}{T}
\newcommand{\tabExtra}[1]{{T}_{#1}}
\newcommand{\len}{\ell en}
\newcommand{\inv}{{\sf Inv}}
\newcommand{\col}{\mathrm{col}}

\newcommand{\semi}{\mathsf{S}}
\newcommand{\nilp}{\mathsf{N}}
\newcommand{\mx}{\mathsf{X}}
\newcommand{\maxset}{\mathcal{M}_{\lambda,\hess}}
\newcommand{\imax}{i_{\mathrm{max}}}

\newcommand{\Poin}{\mathsf{Poin}}

\newcommand{\Fl}{\mathcal{F}\ell ags}

\newcommand{\hinv}[1]{{\sf Inv}(#1)}

\newcommand{\Lint}[1]{{\sf int}_{#1}}
\newcommand{\swap}[1]{{\sf swap}(#1)}
\newcommand{\hess}{{\mathbf h}}

\begin{document}

\title[Dimensions of type $A$ Hessenberg varieties over a fixed sheet]{Dimensions of type $A$ Hessenberg varieties over a fixed sheet}

\author{Megumi Harada}
\address{Department of Mathematics and
Statistics\\ McMaster University\\ 1280 Main Street West\\ Hamilton, Ontario L8S4K1\\ Canada}
\email{haradam@mcmaster.ca}
\urladdr{\url{http://www.math.mcmaster.ca/~haradam/}}

\author{Martha Precup}
\address{Department of Mathematics\\ Washington University in St. Louis \\ One Brookings Drive \\ St. Louis, Missouri  63130 \\ USA }
\email{martha.precup@wustl.edu}
\urladdr{\url{https://www.math.wustl.edu/~precup/}}

\author{Colleen Robichaux}
\address{Department of Mathematics\\ University of California, Davis \\ One Shields Ave.\\ Davis, California 95616\\ USA  }
\email{robichaux@ucdavis.edu}
\urladdr{\url{https://www.math.ucdavis.edu/~robichaux/}}

\keywords{Hessenberg varieties, flag varieties, Springer fibers, Young tableaux}
\date{August 26, 2026}

\begin{abstract} 
Hessenberg varieties $\mathcal{H}\mathrm{ess}(\mathsf{X},\mathbf{h})$ are subvarieties of the flag variety parameterized by a Hessenberg function $\mathbf{h}: [n] \to [n]$ and a matrix $\mathsf{X} \in \mathfrak{gl}_n(\mathbb{C})$. In recent work, Goldin and the second author showed the existence of flat degenerations of Hessenberg varieties to nilpotent Hessenberg varieties over the minimal sheet. This implies that all Hessenberg varieties over the minimal sheet have the same dimension.  Our main result generalizes this dimension result to arbitrary sheets. Specifically, we prove that for a fixed Hessenberg function $\mathbf{h}:[n] \to [n]$, all Hessenberg varieties $\mathcal{H}\mathrm{ess}(\mathsf{X},\mathbf{h})$ defined in the type $A$ flag variety by linear operators $\mathsf{X}$ from the same sheet of the Lie algebra $\mathfrak{gl}_n(\mathbb{C})$ have the same dimension.
\end{abstract}

\maketitle


\section{Introduction}
Let $n$ be a positive integer and $\Fl(\C^n)$ denote the flag variety of nested sequences $V_\bullet=(0\subset V_1\subset V_2 \subset \cdots \subset V_n = \C^n)$ such that $\dim (V_i)=i$ for all $i$. Such a sequence $V_\bullet$ is called a \emph{full flag}. Let $\hess: \{1,2,\ldots, n\} \to \{1,2,\ldots n\}$ be a weakly increasing function such that $\hess(i)\geq i$ for all $i$, called a \emph{Hessenberg function}. The subvariety $\Hess(\mx, \hess)$ of flags $V_\bullet$ such that $\mx (V_i)\subseteq V_{\hess(i)}$ is the (type $A$) \emph{Hessenberg variety} in $\Fl(\C^n)$ associated to the Hessenberg function $\hess$ and matrix $\mx$. Their study lies in the intersection of algebraic geometry, representation theory, and combinatorics,
among other research areas.

Much about the geometry and topology of Hessenberg varieties remains unknown, including the relationships between Hessenberg varieties with different choices of parameters $\mx$ and $\hess$. 
The primary goal of this manuscript is to study the dimensions of Hessenberg varieties for a fixed Hessenberg function $\hess: [n] \to [n]$ as we vary $\mx \in \mathfrak{gl}_n(\C)$. This was motivated in part by recent work of Goldin and the second author ~\cite{Goldin-Precup}, in which they prove the existence of a flat degeneration for any Hessenberg variety defined over $\mx$ in the minimal sheet $\fg_{(2,1,\ldots, 1)}$ to a nilpotent Hessenberg variety defined over the same sheet (cf. Definition~\ref{definition.sheet}). In particular, their result proves that $\dim \Hess(\mx,\hess)$ is the same for all $\mx\in \fg_{(2,1,\ldots 1)}$. It is also known that the Hessenberg varieties defined over the regular sheet $\fg_{(n)}$ of all conjugacy classes of maximum dimension form a flat family~\cite{ADGH}. This led the second author and Goldin to conjecture that a specific one-parameter family of Hessenberg varieties defined over any sheet $\fg_\lambda$ for a fixed Hessenberg function $\hess$ is flat, and in particular, that all Hessenberg varieties defined over a fixed sheet $\fg_\lambda$ with a fixed Hessenberg function $\hess$ have the same dimension~\cite[Conjecture 3.15]{Goldin-Precup}.

It is instructive to consider the following well-known special case. 
The collection of Hessenberg varieties contains the well-studied \emph{Grothendieck--Springer fibers} $\cb_\mx := \Hess(\mx, \hess)$ where $\hess(i)=i$ for all $i$. By definition, the variety $\cb_\mx$ is the subvariety of all flags $V_\bullet$ fixed by $\mx$, that is, $\mx(V_i)\subseteq V_i$ for all $i$.  These varieties are fibers of the Grothendieck--Springer resolution, which plays a prominent role in geometric representation theory and the Springer correspondence.

In the Grothendieck--Springer case, facts about dimensions are known, as we now recall. 
Let $\lambda = (\lambda_1\geq \lambda_2 \geq \cdots \geq \lambda_\ell >0)$ be a partition of $n$ and let $\lambda^t$ denote its transpose. The number 
\begin{equation}\label{eqn.Springer.equality}
n(\lambda):= \sum_{i=1}^\ell (i-1)\lambda _i = \sum_{i=1}^\ell {\lambda_i^t \choose 2}
\end{equation}
is a ubiquitous statistic in algebraic combinatorics and arises in the study of Hall--Littlewood and Macdonald polynomials. The equality~\eqref{eqn.Springer.equality} also has a geometric context. Let $\nilp_\lambda$ be an $n\times n$ nilpotent matrix of Jordan type $\lambda$ and let $\semi_{\lambda}$ be a semisimple matrix with $\ell$ distinct eigenvalues $c_1, c_2, \ldots, c_\ell$ such that $c_i$ occurs with multiplicity $\lambda_i$. Then $\dim \cb_{\nilp_\lambda} = \sum_{i=1}^\ell (i-1)\lambda_i$ and $\dim \cb_{\semi_{\lambda}} = \sum_{i=1}^{\ell}{\lambda_i \choose 2}$. In particular, the formula~\eqref{eqn.Springer.equality} takes the form
\begin{equation}\label{eqn.Springer.equality2}
\dim \cb_{\nilp_\lambda} = \dim \cb_{\semi_{\lambda^t}}.
\end{equation}
The matrices $\nilp_\lambda$ and $\semi_{\lambda^t}$ both belong to a larger union of conjugacy classes in $\mathfrak{gl}_n(\C)$ called a \emph{sheet}. Each sheet in $\mathfrak{gl}_n(\C)$ is indexed by a partition $\lambda$ of $n$ and denoted here by $\fg_\lambda$. Using Springer theory, work of Lusztig~\cite[\S 2.1-2.3, \S 3.2-3.3]{Lusztig} implies that $\dim \cb_\mx=n(\lambda)$ for all matrices $\mx$ in the sheet $\fg_\lambda$. In particular, the dimension of the Grothendieck--Springer fiber over any $\mx\in \fg_\lambda$ is equal to that of any other.

Our main result generalizes the above statement for Grothendieck--Springer fibers to all Hessenberg functions, providing a partial answer to the conjecture of Goldin and the second author.

\begin{thm}[Theorem~\ref{thm.main.general} below]\label{thm.main.intro} Let $n$ be a positive integer and let $\lambda$ be a partition of $n$. Let $\hess: [n] \to [n]$ be a Hessenberg function. Then the dimension of all Hessenberg varieties associated to $\hess$ over the fixed sheet $\fg_\lambda$ is constant.  More precisely, for all $\mx,\mathsf{Y}\in \fg_\lambda$, we have $\dim \Hess(\mx,\hess) = \dim \Hess(\mathsf{Y},\hess)$. 
\end{thm}

The proof of Theorem~\ref{thm.main.intro} is combinatorial, leveraging formulas for the Betti numbers of Hessenberg varieties due to Tymoczko~\cite{Tymoczko2006} and the second author~\cite{Precup2013}. After reviewing the necessary background and establishing notation in Section~\ref{sec.notation}, we prove Theorem~\ref{thm.main.intro} for the nilpotent and semisimple cases in Section~\ref{sec.semisimple and nilpotent Hessenberg varieties} before extending it to the general setting in Section~\ref{sec.general case}. It would be interesting to find a geometric or representation-theoretic argument for this dimension phenomenon. Section~\ref{sec.open} presents related open questions in this direction.

\subsection*{Acknowledgments}  The authors are grateful to Rebecca Goldin for helpful conversations. The first author is partially supported by NSERC Discovery Grant RGPIN 2019-06567 and a Canada Research Chair Tier I Award. The second author is supported by NSF CAREER grant DMS 2237057. The third author was supported by NSF MSPRF No.~DMS 2302279.  We thank the Fields Institute for Research in the Mathematical Sciences in Toronto for hosting the 2022 Virtual Workshop for Women in Commutative Algebra and Algebraic Geometry, from which this project emerged.


\section{Notation and preliminaries}\label{sec.notation}

In this background section we establish notation and conventions.

\subsection{Tableaux, the symmetric group $S_n$, and associated combinatorics}\label{sec:comb}

Fix a positive integer $n$.
A \emph{composition $\gm$ of $n$}, written $\gm \vDash n$, is a list of positive integers $\gm=(\gm_1, \gm_2, \ldots, \gm_k)$ such that $n=\gm_1+\gm_2+\cdots + \gm_k$. In the notation of the previous sentence, we call $k$ the \emph{length of $\gm$}, written as $k= \len(\gm)$. We take the notational convention that $\gm_0 := 0$, and if $i>\len(\gm)$ then $\gm_i =0$.  For a composition $\gm=(\gm_1, \gm_2, \ldots, \gm_k)$ of length $k$, we define 
\begin{equation}\label{eq: def int s} 
\Lint{s}(\gm) :=\left\{i\in\mathbb{Z}  :  \sum_{j=0}^{s-1}\gm_j< i\leq   \sum_{j=0}^{s}\gm_j\right\}
\end{equation}
for $1 \leq s \leq k$. For example, if $\gm= (2,3)$, then
$\Lint{1}(\gm)=\{1,2\}$ and $\Lint{2}(\gm) = \{3,4,5\}$.

A \emph{partition $\lambda$ of $n$}, written $\lambda\vdash n$, is a composition of $n$ such that the list $\lambda=(\lambda_1, \lambda_2, \ldots, \lambda_k)$ satisfies $\lambda_1\geq \lambda_2\geq \cdots \geq \lambda_k$.
The \emph{Young diagram of shape $\lambda$} is a collection of boxes arranged into rows and columns corresponding to the parts in the partition $\lambda$.  We orient our diagrams using the English convention, so the rows are left justified with sizes decreasing from top to bottom.  We denote the transpose of $\lambda$ by~$\lambda^t$. 

Let $\lambda \vdash n$. A \emph{tableau} of shape $\lambda$ and content $[n]:=\{1,2,\cdots,n\}$ is a filling of each box of the Young diagram of $\lambda$ with an integer in $[n]$, with no repetition. We say a tableau is \textit{column increasing} if the entries within each column increase from top to bottom. The \textit{base filling} of $\lambda$ is the tableau obtained by filling the boxes of $\lambda$ with the integers $1$ through $n$ by
starting at the bottom of the leftmost column and moving up the column by increments of $1$, then moving to the lowest
box of the next column to the right, and so on.  For example, the base filling of $\lambda = (4,3,1)$ is the following:
\begin{equation}\label{eqn.base}
\ytableausetup{centertableaux} \begin{ytableau}3 & 5 & 7 & 8\\ 2 & 4 & 6\\ 1\end{ytableau}.
\end{equation}

Let $w \in S_n$ be a permutation of $[n]$. We specify $w$ by its one-line notation $w = [w(1) \, w(2) \cdots \, w(n)]$.
Associated to $w$ and a Young diagram of $\lambda$, we define the tableau $\tabExtra{w,\lambda}$ by relabeling $i$ in the base filling of $\lambda$ to be $w^{-1}(i)$. For example, for $\lambda = (4,3,1)$ and $w = [5 \, 8 \, 4\, 7 \, 3 \, 6 \, 2\, 1]$, we have $w^{-1} = [8 \, 7\, 5 \, 3 \, 1\, 6 \, 4 \, 2]$ and hence we obtain  
\begin{equation}\label{example.TwLambda}
\tabExtra{w,\lambda} = \ytableausetup{centertableaux} \begin{ytableau}5 & 1 & 4 & 2\\ 7 & 3 & 6\\ 8\end{ytableau}. 
\end{equation}

Throughout this paper, we will work with the data of a \emph{Hessenberg function} which is by definition a function $\hess: [n] \to [n]$ such that $\hess(i)\leq \hess(i+1)$ for all $i \in [n-1]$ and $\hess(i) \geq i$ for all $i \in [n]$. We typically notate a Hessenberg function by listing its values in sequence, so we write $\hess = (\hess(1),\hess(2),\cdots,\hess(n))$. 

We will need the following terminology. 

\begin{definition}\label{definition.h-strict}
Let $n$ be a positive integer, $w \in S_n$, and $\lambda \vdash n$ a partition of $n$. 
Suppose $\hess: [n] \to [n]$ is a Hessenberg function. We say that the tableau $\tabExtra{w,\lambda}$ is an \emph{$\hess$-strict tableau} if, for all  $i,j \in [n]$ which appear in horizontally adjacent boxes in $\tabExtra{w,\lambda}$ with $i$ to the left of $j$, we have that $i \leq \hess(j)$. We can informally (and visually) represent this condition as: 
\begin{equation}\label{eq:h-strict}
\ytableausetup{centertableaux} \begin{ytableau}i & j \end{ytableau} \Leftrightarrow i \leq \hess(j). 
\end{equation}
\end{definition}

\begin{example}
For $\hess=(2,4,4,5,6,7,8,8)$, we see that the tableau in~\eqref{example.TwLambda} is \emph{not} $\hess$-strict because we have, for instance, $\begin{ytableau} 7 & 3 \end{ytableau}$ appearing in $\tabExtra{w,\lambda}$, but $7 \not \leq \hess(3)=4$. 
\end{example} 

Recall that a pair $(i,j) \in [n]^2$  is called an \emph{inversion} of a permutation $w\in S_n$ if $i<j$ and $w(i)>w(j)$. We denote the set of inversions of $w$ by $\inv(w)$. 
The \emph{(Bruhat) length} $\ell(w)$ of $w$ is the number of inversions of $w$, or more precisely, 
$$
\ell(w) := \#\inv(w)=\#\{(i,j) \in [n]^2: i<j,\, w(i)>w(j)\} .
$$
Given a Hessenberg function $\hess: [n] \to [n]$, we also define the \emph{Hessenberg length} $\ell_\hess(w)$ of $w \in S_n$ to be 
$$
\ell_\hess(w) := \#\{ (i,j)\in \inv(w) :
j\leq \hess(i)\}.
$$
From the definitions it is clear that $\ell_\hess(w) \leq \ell(w)$ for any $\hess$ and $w$. We will also need to refer to a set naturally enumerating $\ell_\hess(w)$ so let
$$
\hinv{\hess,w} := \{  (i,j)\in \inv(w) :
j\leq \hess(i)\}
$$
be the \emph{$\hess$-inversions of $w$}.
Then, of course, $\ell_\hess(w) = \# \hinv{\hess,w}$. 

Before moving on, we recall the notion of a parabolic subgroup and coset decompositions in the symmetric group. 
Given a composition $\gm = (\gm_1,\gm_2, \ldots, \gm_k)$ of $n$, let 
\begin{equation}\label{eq: D mu}
D_\gm := \{\gm_1, \gm_1+\gm_2 , \ldots, \gm_1+\gm_2+\cdots +\gm_{k-1}\}\subseteq [n-1].
\end{equation}
We define the parabolic subgroup of $S_n$ associated to $\gm$ by 
\begin{equation}\label{eq: parabolic subgroup}
W_\gm := \langle s_i : i\in[n-1]\setminus D_\gm \rangle,
\end{equation}
where $s_i\in S_n$ denotes the simple transposition exchanging $i$ and $i+1$ (and acting as the identity on all other elements of $[n]$). 
Next, we define
\begin{equation}\label{eq: shortest right coset reps}
{^\gm}W := \textup{ shortest right coset representatives of the right cosets of $W_\gm\backslash W$}.
\end{equation}
It is well known that a permutation $w$ is contained in ${^\gm}W$ if and only if the descents of $w^{-1}$ are contained in $D_\gm$, that is:
\[
w\in {^\gm}W \Leftrightarrow \{i\in [n-1] : w^{-1}(i)>w^{-1}(i+1)\} \subseteq D_\gm.
\]
Another way to say this is that $w\in {^\gm}W $ if and only if the numbers from each of the sets $\Lint{s}(\gm)$ with $1\leq s \leq k$ appear in increasing order in the one-line notation for $w$ (when reading from left to right). 

\begin{example}
If $n=5$ and $\gm=(2,1,2)$, then $D_\gm = \{2, 3\}$ and $W_\gm = \langle s_1, s_4\rangle \simeq S_2\times S_2$. The permutation $w=[3,1,4,2,5]$ is an element of ${^\gm}W$ since $w^{-1}=[2,4,1,3,5]$ has descent set $\{2\}$. Equivalently, $w\in {^\gm}W$ because $\Lint{1}(\gm)=\{1,2\}$ and, respectively, $\Lint{3}(\gm)=\{4, 5\}$ appear in increasing order when we read the one-line notation for $w$ from left to right.
\end{example}

The following two lemmas are well-known properties of coset decompositions and length additive decompositions; see \cite[Exercise 13, pg.~23; Prop.~2.4.4]{BB05}.

\begin{lemma}\label{lem.cosets}  Let $w\in S_n$ and let $\gm\vDash n$. Then there exist a unique pair $y,v$ with $y \in W_\gm$ and $v \in {}^{\gm}W$ such that $w = yv$. Moreover, $\ell(w) = \ell(y)+\ell(v)$. 
\end{lemma}

\begin{lemma}\label{lem.lengthadd} Let $w\in S_n$ and suppose $x,z \in S_n$ satisfies $w=xz$ and $\ell(w) = \ell(x)+\ell(z)$. Then the set of inversions of $w$ decomposes as $\inv(w) = \inv(z)\sqcup z^{-1}\,\inv(x)$.
\end{lemma}

Applying Lemma~\ref{lem.lengthadd} to the coset decomposition $w=yv$ from Lemma~\ref{lem.cosets} we have $\inv(w) = \inv(v)\sqcup v^{-1}\inv(y)$. When $\gm = \lambda$ is a partition, we can use the tableau $\tabExtra{w,\lambda}$ to obtain a description of the corresponding decomposition of the inversions of $w$.  

\begin{lemma}\label{lem.cosets.tab} For any $w\in S_n$ and partition $\lambda\vdash n$, the set of inversions $\inv(w)$ of $w$ are precisely the pairs $(i,j) \in [n]^2$ with $i<j$ in $\tabExtra{w,\lambda}$ such that $j$ appears below $i$ in the same column or $j$ appears in any column to the left of the column containing $i$ in $\tabExtra{w,\lambda}$. Furthermore, if $w=yv$ with $y\in W_{\lambda^t}$ and $v\in {}^{\lambda^t} W$ then:
\begin{enumerate}
\item \label{eq: length y from tableau} $(i,j)\in v^{-1}\,\inv(y)$ if and only if $j$ appears below $i$ in the same column of $\tabExtra{w,\lambda}$, and 
\item \label{eq: length v from tableau} $(i,j)\in \inv(v)$ if and only if $j$ appears in any column strictly to the left of the column containing $i$ in $\tabExtra{w,\lambda}$. 
\end{enumerate}
\end{lemma}

\begin{proof}
Recall that, by definition, we construct $\tabExtra{w,\lambda}$ by filling the columns of the Young diagram of shape $\lambda$ with $w^{-1}$, from left to right, and bottom to top. Given a pair $(i,j)$ with $i<j$ such that $j$ appears below $i$ in the same column or $j$ appears in any column to the left of the column containing $i$ in $\tabExtra{w,\lambda}$, it follows immediately that $i = w^{-1}(b)$ and $j = w^{-1}(a)$ with $a<b$, where $a,b$ are the labels of the corresponding boxes in the base filling of shape $\lambda$. Thus, $(i,j)\in \inv(w)$ and every inversion of $w$ is obtained from such a pair in the tableau $\tabExtra{w,\lambda}$.

We now prove~\eqref{eq: length y from tableau}. Suppose that $c<d$ and $y(c)>y(d)$, that is, suppose $(c,d)\in \inv(y)$. Let $a:=y(c)$ and $b:= y(d)$.  Note that since $y \in W_{\lambda^t}$, we must have $a,b, c, d\in \Lint{s}(\lambda^t)$ for the same $s\in[\ell]$.  Since $w=yv$ where $v$ is the shortest coset representative, the assumption that $c,d\in \Lint{s}(\lambda^t)$ with $c<d$ ensures that $v^{-1}(c)<v^{-1}(d)$ and thus $w^{-1}(a)<w^{-1}(b)$. Since $\tabExtra{w,\lambda}$ is filled by $w^{-1}$ and since $a,b\in \Lint{s}(\lambda^t)$, we know that $w^{-1}(a)$ and $w^{-1}(b)$ both lie in column $s$ of $\tabExtra{w,\lambda}$. Moreover, since $a>b$ and we fill $\tabExtra{w,\lambda}$ from bottom to top along columns, we know $w^{-1}(b)=v^{-1}(d)$ lies below $w^{-1}(a)=v^{-1}(c)$ in this column. Thus, $(v^{-1}(c), v^{-1}(d))$ is an element of $\inv(w)$. 
This argument may be reversed and it follows that all pairs in the set $v^{-1}\inv(y)$ are obtained in this way, as was to be shown. 

Statement~\eqref{eq: length v from tableau} now follows immediately. Indeed,  $\inv(w)=\inv(v)\sqcup v^{-1}\inv(y)$ and by the previous paragraphs, the pairs $(i, j)$ from $\inv(w)$ in the set $v^{-1}\inv(y)$ are precisely those in the same column.  We conclude the pairs $(i,j)$ in $\inv(w)$ from different columns must be the elements of $\inv(v)$, as desired.
\end{proof}

\begin{example} Consider $w=[5\ 8\ 4\ 7\ 3\ 6\ 2\ 1]$ and $\lambda = (4,3,1)$. The tableau $\tabExtra{w,\lambda}$ appears in ~\eqref{example.TwLambda} above. The coset decomposition of $w$ relative to $\lambda^t = (3,2,2,1)$ is $w=yv$ where $y = [3 \ 2 \ 1 \ 5 \ 4 \ 7 \ 6 \ 8]$ and $v = [ 4 \ 8 \ 5 \ 6 \ 1 \ 7 \ 2 \ 3]$. The pair $(1,7)$ satisfies the condition that $7$ appears in a column to the left of the column containing $1$ in $\tabExtra{w,\lambda}$. Since these numbers come from distinct columns, $(1,7)\in \inv(v)$, which is easily verified by looking at the one-line notation for $v$. On the other hand, the pair $(4,6)$ satisfies the condition that $6$ appears in the same column as $4$ and below, and we see that $(v(4), v(6)) = (6,7) \in \inv(y)$. 
\end{example}

We can also use the tableau $\tabExtra{w, \lambda}$ to compute the Hessenberg inversions of the corresponding shortest coset representative. Lemma~\ref{lem.cosets.tab} immediately implies the following, from the definitions of $\ell(y)$ and $\ell_\hess(v)$. 

\begin{lemma}\label{lemma.tabSS cardinality}
Let $w \in S_n$ and $\lambda\vdash n$. Let $w = yv$ for $y \in W_{\lambda^t}$ and $v \in {}^{\lambda^t} W$. Then 
\begin{enumerate}
\item $\ell(y)$ is the number of pairs $(i,j) \in [n]^2$ with $i<j$ such that $j$ appears below $i$ in the same column of $\tabExtra{w,\lambda}$, and 
\item $\ell_\hess(v)$ is the number of pairs $(i,j) \in [n]^2$ with $i<j\leq \hess(i)$ such that $j$ appears in any column strictly to the left of the column containing $i$ in $\tabExtra{w,\lambda}$.
\end{enumerate}
\end{lemma}

\begin{example}\label{ex:invTab} Suppose $w=[5\, 2\, 3\, 8\, 6 \, 7\, 1\, 4]$ and $\lambda=(4,3,1)$. We have $\lambda^t = (3,2,2,1)$ and we can decompose $w$ as $w=yv$, where we see $y=[2\, 3\,1\,5\,4\,6\,7\,8]\in W_{\lambda^t}$ and $v=[4\,1\,2\,8\,6\,7\,3\,5]\in {^{\lambda^t} W}$. The corresponding tableau is displayed below.
\[
\tabExtra{w,\lambda}=\ytableausetup{centertableaux} \begin{ytableau}
3 & 1 & 6 & 4\\
2 & 8 & 5 \\ 
7 
\end{ytableau}
\]
Then using Lemma~\ref{lem.cosets.tab}, we compute
\begin{align*}
    \ell(y)&=3= \# \left(v^{-1}\Inv(y)\right)=\#\{(2,7),(3,7),(1,8)\}\\
    \ell(v)&=11=\#\{(1,2),(1,3),(1,7),(4,5),(4,6),(4,7),(4,8),(5,7),(5,8),(6,7),(6,8)\}
\end{align*}
Now suppose $\hess=(2,3,5,5,7,7,8,8)$; Lemma~\ref{lemma.tabSS cardinality} implies
\[
\ell_\hess(v) = 4 = \#\{(1,2),(4,5),(5,7),(6,7)\}.
\]
\end{example}


\subsection{Type $A$ Hessenberg varieties and sheets}
We now define some of the geometric objects which are studied in this paper. We focus exclusively on the case of Lie type $A$, i.e., the case of the algebraic group ${ GL}_n(\mathbb{C})$ with Lie algebra $\mathfrak{gl}_n(\C)$. 

Recall that the flag variety $\Fl(\C^n)$ is the variety of nested sequences of subspaces 
$$
\Fl(\C^n)=\{V_\bullet = (0  \subset V_1 \subset V_2 \subset \cdots \subset V_n=\C^n) :  \dim_\C(V_i) = i  \, \textup{ for all } \, i \in [n] \}. 
$$
Now let $\mx \in \mathfrak{gl}_n(\C)$ be an $n \times n$ matrix with complex entries and let $\hess: [n] \to [n]$ be a Hessenberg function. 
Then the \emph{Hessenberg variety associated to $\mx$ and $\hess$} is defined to be 
$$
\Hess(\mx,\hess) := \{V_\bullet \in \Fl(\C^n) : \mx V_i \subseteq V_{\hess(i)} \, \textup{ for all } \, i \in [n]\}.
$$

As is clear from the definition, Hessenberg varieties are parametrized in part by   $\mathfrak{gl}_n(\C)$ through the choice of $\mx \in \mathfrak{gl}_n(\C)$. The following is well-known and straightforward to prove.

\begin{lemma}\label{lemma.conjugation} If $\mx$ and $\mathsf{Y}$ are conjugate matrices, then there is an isomorphism of the corresponding Hessenberg varieties $\Hess(\mx,\hess)\simeq \Hess(\mathsf{Y},\hess)$.
\end{lemma}

When $\mx$ and $\mx'$ are from different conjugacy classes, the geometry of the corresponding Hessenberg varieties may differ significantly. Part of the motivation of this manuscript is to study how different Hessenberg varieties are related, as we pick matrices from different conjugacy classes. To study this question, we next define certain subsets of $\mathfrak{gl}_n(\C)$ parametrized by partitions $\lambda$. 

Let $\mx\in \mathfrak{gl}_n(\C)$ with distinct eigenvalues $\{c_1,c_2, \ldots, c_m\}$. A \emph{generalized Jordan block} of $\mx$ consists of all Jordan blocks of $\mx$ associated to a single eigenvalue $c_i$ of $\mx$. Each generalized Jordan block determines a partition, which we denote $\lambda^{(i)}$, recording the sizes of the Jordan blocks for $c_i$ in decreasing order. Thus, if $\lambda^{(i)} = (\lambda_1^{(i)}, \lambda_2^{(i)}, \ldots, \lambda_k^{(i)})$ then the algebraic multiplicity of $c_i$ is $|\lambda^{(i)}|:=\lambda_1^{(i)}+\lambda_2^{(i)}+\cdots \lambda_k^{(i)}$.   
Now let $\len(\lambda^{(i)})$ be the length of $\lambda^{(i)}$ and set $\ell := \max_{1 \leq i \leq m} \{\len(\lambda^{(i)})\}$. We define a partition $\lambda_\mx= (\lambda_{\mx,1}, \lambda_{\mx,2}, \ldots, \lambda_{\mx,\ell})\vdash n$
by setting $\lambda_{\mx,j} := \sum_{i=1}^m \lambda_j^{(i)}$ for each $j\in [\ell]$. 

\begin{definition}\label{definition.sheet}
Following the notation just established, the \emph{sheet in $\mathfrak{gl}_n(\C)$ associated to the partition $\lambda\vdash n$} is then defined as 
\[
\fg_\lambda := \{\mx \in \mathfrak{gl}_n(\C) : \lambda_\mx=\lambda\}.
\]
\end{definition}

\begin{example}\label{ex.sheet} Consider 
\[
\mx= \begin{bmatrix} 0 & 1 & 0 & 0 \\ 0 & 0 & 0 & 0\\ 0 & 0 & 1 & 0 \\ 0 & 0 & 0 & 1 \end{bmatrix} .
\]
We have two eigenvalues, $c_1=0$ and $c_2=1$ with $\lambda^{(1)} = (2)$ and $\lambda^{(2)}=(1,1)$. Thus $\ell=2$, $\lambda_\mx = (3,1)$, and $\mx\in \fg_{(3,1)}$. 
\end{example}

The following is immediate from the definition of $\fg_\lambda$ above.

\begin{lemma} \label{lemma: orbits in g lambda}
Let $\lambda$ be a partition of $n$. 
\begin{enumerate}
\item The sheet $\fg_\lambda$ contains a unique nilpotent orbit, namely the orbit of all nilpotent matrices of Jordan type $\lambda$. 
\item A semisimple matrix $\semi$ is in $\fg_\lambda$ if and only if it has $\lambda_1$ many distinct eigenvalues $c_1,\cdots,c_{\lambda_1}$, where the algebraic multiplicity of $c_i$ is equal to $\lambda^t_i$ for each $i$. 
\end{enumerate}
\end{lemma}


\subsection{Affine pavings of type $A$ Hessenberg varieties}\label{subsec: affine pavings}

We now turn our attention to affine pavings of Hessenberg varieties $\Hess(\mx,\hess)$. These pavings will be the main tool by which we obtain our main results. Known constructions of affine pavings of Hessenberg varieties utilize the well-known affine paving of $\Fl(\C^n)$ by Schubert cells. 
Recall that the Schubert cell is defined as $C_w :=B\dot{w}B/B$, where $B\subset { GL}_n(\mathbb{C})$ is the Borel subgroup of upper triangular matrices and $\dot{w}$ is the permutation matrix for $w\in S_n$ such that column $i$ has unique nonzero entry equal to $1$ in row $w(i)$.  The intersection $C_w\cap \Hess(\mx, \hess)$ is called a \emph{Hessenberg--Schubert cell}.

In Section~\ref{sec.semisimple and nilpotent Hessenberg varieties} below, we focus on two specific types of matrices $\mx$, namely, the nilpotent case and the semisimple case. This is due to the fact that more is known about the combinatorics governing the Betti numbers of the Hessenberg varieties, as well as the affine pavings, associated to these cases.  When $\mx$ is nilpotent, the corresponding Hessenberg variety $\Hess(\mx, h)$ is called a \emph{nilpotent Hessenberg variety} and, similarly, if $\mx$ is semisimple then $\Hess(\mx, h)$ is a \emph{semisimple Hessenberg variety}. 

We begin with the nilpotent case. Let $\lambda\vdash n$. Following Tymoczko~\cite{Tymoczko2006}, we fix a specific nilpotent matrix $\nilp_\lambda$ of Jordan type $\lambda$ as follows. 

\begin{definition}\label{def.base.nilp} For each partition $\lambda \vdash n$ define $\nilp_\lambda := \sum E_{\ell r}$, where the sum is over all pairs $(\ell,r)$ such that  $r$ labels the box directly to the right of $\ell$ in the base filling of $\lambda$. 
\end{definition}

Although $\nilp_\lambda$ is not in Jordan canonical form, it is not hard to see that $\nilp_\lambda$ is nilpotent and $\nilp_\lambda \in \fg_\lambda$. For example, when $\lambda = (4,3,1)$ with base filling as in~\eqref{eqn.base} we have
\[
\nilp_\lambda =  \begin{bmatrix} 0 & 0 & 0  & 0 & 0 & 0 & 0 & 0\\ 0 & 0 & 0 & 1 & 0 & 0 & 0 & 0\\ 0 & 0 & 0 & 0 & 1 & 0 & 0 & 0\\ 0 & 0 & 0 & 0 & 0 & 1 & 0 & 0\\ 0 & 0 & 0 & 0 & 0 & 0 & 1 & 0\\   0 & 0 & 0 & 0 & 0 & 0 & 0 & 0\\ 0 & 0 & 0 & 0 & 0 & 0 & 0 & 1\\ 0 & 0 & 0 & 0 & 0 & 0 & 0 & 0\\ \end{bmatrix}.
\]
The following is a result of Tymoczko, stated for the special case in which  the matrix $\mx={\sf \nilp_\lambda}$ is nilpotent. We use the notion of $h$-strict tableau, introduced in Definition~\ref{definition.h-strict}.

\begin{prop}[Tymoczko~{\cite[Theorem 7.1]{Tymoczko2006}}] \label{prop.nildim} Let $\lambda$ and $\nilp_\lambda$ be as above. Let $w \in S_n$ and $\hess: [n] \to [n]$ a Hessenberg function. Then the Hessenberg--Schubert cell $C_w\cap \Hess(\nilp_\lambda,\hess)$ is nonempty if and only if $\tabExtra{w,\lambda}$ is $\hess$-strict. If $\tabExtra{w,\lambda}$ is $\hess$-strict, then $C_w \cap \Hess(\nilp_\lambda,\hess)$ is isomorphic to an affine space, and its dimension $\dim(C_w\cap \Hess(\nilp_\lambda,\hess))$ is given by counting the number of pairs $(i,j)\in [n]^2$ where $i<j$ such that 
\begin{enumerate}
\item $j$ appears below $i$ in the same column or $j$ appears in any column to the left of the column containing $i$ in $\tabExtra{w,\lambda}$, and
\item if $k$ labels the box directly to the right of $i$ in $\tabExtra{w,\lambda}$, then $j\leq \hess(k)$.
\end{enumerate}
Moreover, the collection $\{C_w \cap \Hess(\nilp_\lambda, \hess): w \in S_n, \,\tabExtra{w,\lambda} \; \hess\text{-strict} \}$ is an affine paving of $\Hess(\nilp_\lambda, \hess)$.
\end{prop}

Note that item (1) in the proposition is equivalent to the statement that $(i,j)\in \Inv(w)$ by Lemma~\ref{lem.cosets.tab}.
Inspired by the above, we introduce the following terminology; note that the set of ``nil-inversions'' defined below differs from the set of $\hess$-inversions $\hinv{\hess,w}$ of $w$ introduced in Section~\ref{sec:comb} above.

\begin{definition}\label{definitiion.h.w.lambda.inversion} 
We call a pair $(i,j) \in [n]^2$ satisfying (1) and (2) from  Proposition~\ref{prop.nildim} a \emph{nil-inversion} of $\tabExtra{w,\lambda}$. 
\end{definition}

Before stating the analogous result in the semisimple setting, we need some notation. 

\begin{definition}\label{definition.semi.tau}
Let $\tau\vDash n$. We define $\semi_\tau$ to be a diagonal $n \times n$ matrix with $\len(\tau)$ distinct eigenvalues $c_1, \cdots, c_{\len(\tau)}$, listed in that order along the diagonal, with the algebraic multiplicity of $c_i$ equal to $\tau_i$. (Our arguments are independent of the values of the eigenvalues $c_i$, as long as they are distinct.) 
\end{definition} 

From Lemma~\ref{lemma: orbits in g lambda} it follows that $\semi_\tau$ belongs to $\mathfrak{g}_{(\mathrm{sort}(\tau))^t}$ where $\mathrm{sort}(\tau)$ is the partition obtained from $\tau$ by sorting the parts of $\tau$ in decreasing order.

The following result says the Hessenberg--Schubert cells give an affine paving of the semisimple Hessenberg variety $\Hess(\semi_{\tau},\hess)$. This is a restatement of~\cite[Corollary 7.2]{Tymoczko2006}, where in the notation of the following proposition, $\ell(y)$ counts the pairs from~\cite[Corollary 7.2(1)]{Tymoczko2006} and $\ell_\hess(v)$ counts the pairs from~\cite[Corollary 7.2(2)]{Tymoczko2006}.

\begin{prop}[Tymoczko~{\cite[Corollary 7.2]{Tymoczko2006}}]\label{prop.ssdim} Let $\tau\vDash n$ and $\semi_{\tau}$ be a diagonal matrix as defined in Definition~\ref{definition.semi.tau}. For each $w\in S_n$, write $w=yv$ with $y\in W_{\tau}$ and $v\in {^{\tau}}W$. Then the intersection $C_w \cap \Hess(\semi_\tau, \hess)$ is nonempty, and is isomorphic to an affine space of dimension
\[
\dim (C_w\cap \Hess (\semi_{\tau},\hess)) =  \ell(y)+\ell_\hess(v). 
\]
Moreover, the collection $\{C_w \cap \Hess(\semi_{\tau}, \hess): w \in S_n\}$ is an affine paving of $\Hess(\semi_{\tau}, \hess)$.
\end{prop}

Since the matrices $\semi_\tau$ and $\semi_{\mathrm{sort}(\tau)}$ are conjugate, we have $\Hess(\semi_\tau , \hess) \simeq \Hess(\semi_{\mathrm{sort}(\tau)},\hess)$. As a result we often assume  $\tau$ is a partition when we study semisimple Hessenberg varieties. However, some of our later arguments are simpler in the slightly more general setting.

\begin{example}\label{ex:ssTab1}
Let $w=[5\, 2\, 3\, 8\,6\,7\,1\,4]$ and $\lambda=(4,3,1)$ as in Example~\ref{ex:invTab}. We choose $\tau = \lambda^t = (3,2,2,1)$ and $\hess=(2,3,5,5,7,7,8,8)$. We may decompose $w$ as $w=yv$ with $\ell(y)= 3$ and $\ell_\hess(v)=4$, as computed in Example~\ref{ex:invTab}. By Proposition~\ref{prop.ssdim}, we know 
\[
\dim \left( C_w\cap\Hess(\semi_{\lambda^t},\hess)\right) =\ell(y)+\ell_\hess(v)=3+4=7.
\]
\end{example}

\begin{rem} As stated explicitly in Proposition~\ref{prop.nildim}, there is an affine paving of the nilpotent Hessenberg variety $\Hess(\nilp_\lambda, \hess)$ with cells indexed by $\hess$-strict tableaux. On the other hand, using Proposition~\ref{prop.ssdim}, we obtain an affine paving of $\Hess(\semi_{\lambda^t}, \hess)$ with $n!$ cells, one for each permutation $w\in S_n$.
In order to facilitate the comparison of the dimensions of $\Hess(\nilp_\lambda, \hess)$ and $\Hess(\semi_{\lambda^t}, \hess)$, we will use the bijective correspondence between $w \in S_n$ and tableaux $T_{w,\lambda}$ to index the affine cells of $\Hess(\semi_{\lambda^t},\hess)$ by tableaux. Note that, in the semisimple case, this means we are considering \emph{all} tableaux of the form $\tabExtra{w,\lambda}$, not just those which are $\hess$-strict. 
\end{rem}


\section{Semisimple and nilpotent Hessenberg varieties over a fixed sheet}\label{sec.semisimple and nilpotent Hessenberg varieties}

The main result of this section shows that the dimensions of nilpotent and semisimple Hessenberg varieties over the same sheet are equal. More precisely, we have the following. 

\begin{thm} \label{thm.ss-nilp} Let $\hess:[n]\to [n]$ be a Hessenberg function and $\lambda\vdash n$. Then the semisimple Hessenberg variety $\Hess(\semi_{\lambda^t}, \hess)$ and the nilpotent Hessenberg variety $\Hess(\nilp_\lambda, \hess)$ have the same dimension, i.e., 
\[
\dim \left(\Hess(\nilp_\lambda,\hess)\right) = \dim \left(\Hess(\semi_{\lambda^t},\hess)\right).
\]
In particular, all semisimple and nilpotent Hessenberg varieties defined using the Hessenberg function $\hess$ over the sheet $\fg_\lambda$ have equal dimension.
\end{thm}

We need several preliminaries to prove Theorem~\ref{thm.ss-nilp},  so we begin with a rough outline of our argument.
The central idea is that, since the results of Section~\ref{subsec: affine pavings} give us an affine paving of both $\Hess(\nilp_\lambda,\hess)$ and $\Hess(\semi_{\lambda^t},\hess)$, the dimension of each of these varieties can be obtained by computing the \emph{maximum} dimension of the affine cells in their respective pavings. Moreover, the results of Section~\ref{subsec: affine pavings} show that the affine pavings in question are obtained by intersecting with Schubert cells $C_w$. With this in mind, our argument can be broken down into proving the following:

\begin{prop} \label{prop.steps} 
The following statements hold:
\begin{enumerate} 
\item There is a subset of $S_n$, namely 
\begin{equation}\label{eq: max achieving subset} 
\maxset := \{w \in S_n: T_{w,\lambda} \textup{ is  $\hess$-strict and is column increasing} \}, 
\end{equation}
such that for $w \in \maxset$, the two affine cells $C_w \cap \Hess(\nilp_\lambda,\hess)$ and $C_w \cap \Hess(\semi_{\lambda^t},\hess)$ are both non-empty, and their dimensions are equal. 
\item For any $w \not \in \maxset$ for which $C_w \cap \Hess(\nilp_\lambda,\hess) \neq \emptyset$, we can find some permutation $\tilde{w} \in \maxset$ such that 
$\dim(C_w \cap \Hess(\nilp_\lambda,\hess)) \leq \dim(C_{\tilde{w}} \cap \Hess(\nilp_\lambda, \hess))$. In particular, the maximum dimension of the affine cells $C_w \cap \Hess(\nilp_\lambda,\hess)$ is achieved by a permutation $w$ in the set $\maxset$.  Moreover, if $C_w \cap \Hess(\nilp_\lambda, \hess)$ is maximum-dimensional, then $w \in \maxset$. 
\item For any $w \not \in \maxset$,  we can find a finite sequence of permutations $w =w_0\mapsto w_1 \mapsto w_2 \mapsto \cdots \mapsto w_{\ell}$ such that $w_{\ell} \in \maxset$ and the corresponding sequence of values of
$\dim(C_{w_i} \cap \Hess(\semi_{\lambda^t},\hess))$
is weakly increasing. In particular, the maximum dimension of the affine cells $C_w \cap \Hess(\semi_{\lambda^t},\hess)$ is achieved for $w$ in the set $\maxset$. 
\end{enumerate}
\end{prop}

Putting the above claims together, since statements (2) and (3) show that the maximum dimensions of the affine cells for both $\Hess(\nilp_\lambda,\hess)$ and $\Hess(\semi_{\lambda^t},\hess)$ are achieved on $\maxset$, and statement (1) shows that on $\maxset$ the associated dimensions of the affine cells are equal, we may conclude that these maximum dimensions are the same. This shows that $\dim\left(\Hess(\nilp_\lambda,\hess) \right)=\dim\left(\Hess(\semi_{\lambda^t},\hess)\right)$, which is the first assertion in Theorem~\ref{thm.ss-nilp}. 

In what follows, Lemmas~\ref{lem.h-strict.set-inclusion} and~\ref{lem.h-strict} achieve Proposition~\ref{prop.steps}(1). Proposition~\ref{prop.steps}(2) is achieved by Lemma~\ref{lem.inc-cols.nilp}.  We then prove Lemma~\ref{lem.inc-cols.ss}, Lemma~\ref{lemma.SScells-swap}, Lemma~\ref{lem:swapCol}, and Lemma~\ref{prop:lenSwaps} to obtain Proposition~\ref{prop.steps}(3). 

\medskip

Bearing the sketch from Proposition~\ref{prop.steps} in mind, we begin working in earnest toward a proof of statement (1).

Our first lemma computes and compares codimensions of Hessenberg--Schubert cells of nilpotent and semisimple Hessenberg varieties
combinatorially:

\begin{lemma}\label{lem.h-strict.set-inclusion}
Let $w \in S_n$, $\lambda\vdash n$, and $\hess: [n] \to [n]$ a Hessenberg function. Suppose that the tableau $\tabExtra{w,\lambda}$ is $\hess$-strict. The following hold:
\begin{enumerate}
\item The codimension $\codim(C_w \cap \Hess(\nilp_\lambda,\hess), C_w)$ of $C_w \cap \Hess(\nilp_\lambda,\hess)$ in $C_w$ is the cardinality of the set $\mathcal{T}_{w,\lambda,\hess}$ of pairs $i<j$ such that: 
    \begin{enumerate} 
    \item $j$ appears in the same column and below $i$ in $\tabExtra{w,\lambda}$, or, in any column to the left of the column containing $i$ in $\tabExtra{w,\lambda}$, and
    \item there exists an entry, call it $k$, lying in the box directly to the right of $i$ in $\tabExtra{w,\lambda}$, and $j > \hess(k)$. 
    \end{enumerate} 

\item The codimension $\codim(C_w \cap \Hess(\semi_{\lambda^t},\hess),C_w)$of $C_w \cap \Hess(\semi_{\lambda^t},\hess)$ in $C_w$ is equal to the cardinality of the set $\mathcal{S}_{w,\lambda,\hess}$ of pairs $k<j$ such that $j$ appears in a column strictly to the left of the column containing $k$ in $\tabExtra{w,\lambda}$, and $j>\hess(k)$. 

\item There is an injection of sets $\phi: \mathcal{T}_{w,\lambda,\hess} \hookrightarrow \mathcal{S}_{w,\lambda,\hess}$. 

\item The inequality $\codim(C_w \cap \Hess(\semi_{\lambda^t},\hess),C_w) \geq \codim(C_w \cap \Hess(\nilp_\lambda,\hess), C_w)$ holds.

\item If, in addition, $\tabExtra{w,\lambda}$ is column increasing, there is an injection $\psi: \mathcal{S}_{w,\lambda,\hess} \hookrightarrow \mathcal{T}_{w,\lambda,\hess}$, and $\codim(C_w \cap \Hess(\semi_{\lambda^t},\hess),C_w) \leq \codim(C_w \cap \Hess(\nilp_\lambda,\hess), C_w)$. 
\end{enumerate}
\end{lemma}

\begin{proof} 
First recall the well-known fact that $\dim C_w=\ell(w) = |\inv(w)|$. 
By Lemma~\ref{lem.cosets.tab} $\Inv(w)$ is precisely the set of pairs $(i,j)$ such that $i<j$ and $j$ appears in the same column and below $i$, or, in any column to the left of the column containing $i$ in $\tabExtra{w,\lambda}$. 

 Comparing with Proposition~\ref{prop.nildim}, it follows that 
$\mathrm{codim}(C_w \cap \Hess(\nilp_\lambda,\hess), C_w)$ is obtained by counting the number of pairs $i<j$ such that the following holds: there exists a box directly to the right of $i$, and if we let $k$ denote the entry in this box, then $j > \hess(k)$. This shows (1). 
Comparing next with Lemma~\ref{lemma.tabSS cardinality} and Proposition~\ref{prop.ssdim}, it follows that $\codim(C_w \cap \Hess(\semi_{\lambda^t},\hess),C_w)$ is computed by counting the number of pairs $k<j$ such that $j$ appears in a column strictly to the left of the column containing $k$, and $j > \hess(k)$. This shows (2). 

We next prove (3). Suppose $i<j$ is a pair in $\mathcal{T}_{w,\lambda,\hess}$, i.e., $i<j$ satisfy the conditions in (1)(a--b). Since $k$ appears directly to the right of $i$, it follows that $j$ appears in a column strictly to the left of $k$ in $\tabExtra{w,\lambda}$. Additionally, by condition (1)(b) we know $j>\hess(k)$, and since $\hess(k)\geq k$, it follows that $j>k$. Thus $(k,j)$ is contained in $\mathcal{S}_{w,\lambda,\hess}$. It is straightforward to see that the association $\phi: (i,j) \mapsto (j,k)$ (where $k$ to the right of $i$) is injective. This proves (3). The claim (4) follows directly from (1)-(3).

We now prove the last claim (5). Let $(k,j) \in \mathcal{S}_{w,\lambda,\hess}$.  Let $i$ denote the entry directly to the left of $k$ in $\tabExtra{w,\lambda}$. Such a box exists since the assumptions imply that $k$ is not in the leftmost column. Since $\tabExtra{w,\lambda}$ is $\hess$-strict, we know that $i\leq \hess(k)$, and by assumption $j>h(k)$ so $i<j$. Now, either $j$ is in the same column as $i$ or in a column strictly to the left of $i$. We have assumed that $\tabExtra{w,\lambda}$ is column increasing, so in the former case, $j$ must occur in the same column and below $i$ in $\tabExtra{w,\lambda}$. But this implies that in both the former and the latter case, the pair $i<j$ is contained in $\mathcal{T}_{w,\lambda,\hess}$. It is straightforward to see that this association $\psi: (k,j) \mapsto (i,j)$ is injective. The claim about codimensions immediately follows from the injectivity of $\psi$ and claims (1) and (2). 
\end{proof} 

The above lemma leads to the following, which shows that the dimensions of the affine cells for the nilpotent and semisimple cases agree if $w$ is such that $\tabExtra{w,\lambda}$ is both $\hess$-strict and column-increasing. 

\begin{lemma}\label{lem.h-strict} Let $w \in S_n$, $\lambda\vdash n$, and $\hess: [n] \to [n]$ a Hessenberg function. Suppose that the tableau $\tabExtra{w,\lambda}$ is $\hess$-strict. We have 
\begin{enumerate}
\item $\dim \left( C_w\cap \Hess(\nilp_\lambda,\hess)\right) \geq \dim \left(C_w\cap \Hess(\semi_{\lambda^t},\hess)\right)$, and furthermore,
\item if $\tabExtra{w,\lambda}$ is column increasing, that is, if $\tabExtra{w,\lambda}\in \cm_{\lambda, \hess}$, then $\dim \left(C_w\cap \Hess(\nilp_\lambda,\hess)\right) = \dim \left( C_w\cap \Hess(\semi_{\lambda^t},\hess)\right)$.
\end{enumerate}
\end{lemma}

\begin{proof} 
The first claim is an immediate consequence of Lemma~\ref{lem.h-strict.set-inclusion}(4). For the second claim, Lemma~\ref{lem.h-strict.set-inclusion}(4) and (5) imply that if $\tabExtra{w,\lambda}$ is column increasing then $\codim(C_w \cap \Hess(\semi_{\lambda^t},\hess),C_w) = \codim(C_w \cap \Hess(\nilp_\lambda,\hess), C_w)$. This then implies the second claim. 
\end{proof}

This completes the proof of statement (1) in the outline given in Proposition~\ref{prop.steps} at the beginning of this section. 

\medskip

We next embark on (2) and (3) from Proposition~\ref{prop.steps}. 
First, we need some notation. For $w \in S_n$ and $\lambda\vdash n$, consider the tableau, denoted $\col(\tabExtra{w,\lambda})$, which is  obtained from $\tabExtra{w,\lambda}$ by sorting the entries within each column in increasing order (from top to bottom). Then there exists a unique permutation, which we denote by $\tilde{w}$, for which $\col(\tabExtra{w,\lambda}) =\tabExtra{\tilde{w},\lambda}$. The following two lemmas compare the dimensions of the affine cells corresponding to $\tabExtra{w,\lambda}$ and $\tabExtra{\tilde{w},\lambda}$. We begin with the semisimple case.

\begin{lemma}\label{lem.inc-cols.ss} Let $w \in S_n$ and $\lambda\vdash n$.  Let $\col(\tabExtra{w,\lambda})$ and $\tilde{w}$ be as above.  Then 
\[
\dim \left(C_w\cap \Hess(\semi_{\lambda^t},\hess)\right) \leq \dim \left( C_{\tilde{w}}\cap \Hess(\semi_{\lambda^t},\hess)\right).
\]
\end{lemma}

\begin{proof}
Let $w=yv$ for $y\in W_{\lambda^t}$ and $v\in {^{\lambda^t} W}$. Similarly, set $\tilde{w}=\tilde{y}\tilde{v}$ for $\tilde{y}\in W_{\lambda^t}$ and $\tilde{v}\in {^{\lambda^t} W}$. Since $\col(\tabExtra{w,\lambda}) = \tabExtra{\tilde{w},\lambda}$ is obtained from $\tabExtra{w,\lambda}$ by changing the order of entries within each column, and no swaps of entries between different columns, it follows from 
Lemma~\ref{lem.cosets.tab}(2) that $v=\tilde{v}$ (as permutations are uniquely determined by their inversion sets). Moreover, since $\tabExtra{\tilde{w},\lambda}$ is column increasing by construction, it follows from Lemma~\ref{lem.cosets.tab}(1) that $\ell(y) \leq \ell(\tilde{y})$ (in fact, $\tilde{y}$ is the longest element of $W_{\lambda^t}$). Now the claim follows from Proposition~\ref{prop.ssdim}.
\end{proof}

The next lemma tackles the analogous question for nilpotent Hessenberg varieties, which achieves Proposition~\ref{prop.steps}(2).

\begin{lemma}\label{lem.inc-cols.nilp} Let $w \in S_n$, $\lambda\vdash n$, and $\hess:[n] \to [n]$ a Hessenberg function.  Suppose that $\tabExtra{w,\lambda}$ is $\hess$-strict. Let $\col(\tabExtra{w,\lambda})$ and $\tilde{w}$ be as above. Then
\begin{enumerate} 
\item $\tabExtra{\tilde{w},\lambda}$ is $\hess$-strict, and 
\item if $\tilde{w}\neq w$, then $\dim \left(C_w\cap \Hess(\nilp_\lambda,\hess) \right)< \dim \left(C_{\tilde{w}}\cap \Hess(\nilp_\lambda,\hess)\right)$.
\end{enumerate} 
\end{lemma}

\begin{proof} 
This result follows similarly to \cite[Theorem~6.3]{Ji-Precup2019}, but we include the details here for completeness. We prove the claim by induction on the number of ``column inversions'' in the tableau. More precisely, first set the temporary notation $T :=\tabExtra{w,\lambda}$. Recall that $T$ is $\hess$-strict by assumption. Consider the set 
$$
\mathrm{ColInv}(T) := \{(i,j)\,:\, i<j, \, i \textup{ appears below }\, j \textup{ in the same column as} \, \, j\}. 
$$
We prove the claims (1) and (2) by induction on 
$m := \# \mathrm{ColInv}(T)$. 

The base case is $m=0$, i.e., when $\mathrm{ColInv}(T)$ is empty. 
It is clear that in this case, $T$ is column increasing. Thus, in this case, $\col(T) = T$  and $\tilde{w} = w$, and claim (1) holds. In this case, claim (2) vacuously holds since $\tilde{w}=w$. 

Now suppose that $m>0$, and the claims (1) and (2) are true for $m-1$. In particular, since $m>0$, we know there exists at least one pair $i<j$ where $i$ appears below $j$ in the same column. 
Let $r$ be the largest value for which a column inversion in $T$ appears in rows $r, r+1$. Let $T'$ denote the tableau obtained from $T$ by sorting the entries of $T$ within columns, but only in rows $r$ and $r+1$, to be in increasing order (reading top to bottom). Let $w'$ denote the permutation such that $T'=\tabExtra{w',\lambda}$.
By construction, $T'$ has $m'$ column inversions where $m'<m$. Moreover, note that $\col(T) = \col(T')$, since $T$ and $T'$ have the same entries in each column.

To complete the proof by induction, we claim that it is sufficient to show 
\begin{enumerate} 
\item[(i)] $T'$ is $\hess$-strict, and 
\item[(ii)] the number of nil-inversions of $T'$ is strictly greater than those of $T$.
\end{enumerate} 
Indeed, if $T'$ is $\hess$-strict, then since $m'<m$ we may by induction assume that (1) and (2) hold for $T'$. 
We know $\col(T) = \col(T')=\tabExtra{\tilde{w},\lambda}$ for some $\tilde{w}\in S_n$. Thus the inductive assumption implies $\tabExtra{\tilde{w},\lambda}$ is $\hess$-strict, so (1) holds for $T$. Since (2) holds for $T'$ we have  $\dim \left(C_{w'} \cap \Hess(\nilp_\lambda, \hess)\right) < \dim \left(C_{\tilde{w}} \cap \Hess(\nilp_\lambda, \hess)\right)$. 
If statement (ii) also holds, then since the nil-inversions of $\tabExtra{w,\lambda}$ count the dimension of the affine cells of $\Hess(\nilp_\lambda,\hess)$ by Proposition~\ref{prop.nildim}, we may conclude $\dim \left(C_w \cap \Hess(\nilp_\lambda,\hess)\right) < \dim  \left( C_{w'} \cap \Hess(\nilp_\lambda,\hess)\right)$. Putting these together shows (2) for $w$. Thus, it only remains to prove (i) and (ii). 

To see (i) and (ii), we analyze $2\times 2$ sub-configurations in $T$ as below, where we fix the rows to be $r$, $r+1$ and consider all consecutive pairs of columns $c$, $c+1$:
\[\begin{ytableau}
    i & k_i\\
    j & k_j
\end{ytableau}\]
when we transform $T$ into $T'$.
In addition, we must consider cases that arise when the two rows in question, rows $r$ and $r+1$, have a different number of boxes, so it may happen that we have a sub-configuration of the form 
\[\begin{ytableau}
    i & k_i\\
    j 
\end{ytableau}\]
(so there is no box to the right of $j$). Similarly there is the case when there are no boxes to the right of both $i$ and $j$ so we have 
\[\begin{ytableau}
    i \\
    j 
\end{ytableau}\]

Case 1 analyzes when $(i,j)$ and $(k_i,k_j)$ have the same relative order, and Case 2 analyzes when they are different. Then Case 3 analyzes when there is no box to the right of $j$, and Case 4 analyzes when there is no box to the right of both $i$ and $j$.

\smallskip
\noindent {\sf Case 1a ($i<j$ and $k_i < k_j$)}:   In this case, the resulting $2 \times 2$ configuration is unchanged: 
\[
\begin{ytableau}
    i & k_i\\
    j & k_j
\end{ytableau} \ .
\]
Thus, $\hess$-strictness is preserved in this sub-diagram. We show that the nil-inversions of $T$ involving $i$ or $j$ in the first coordinate are all nil-inversions of $T'$ in this case. If $(i,j)$ was a nil-inversion of $T$, then since the $2 \times 2$ configuration remains unchanged, $(i,j)$ is still a nil-inversion of $T'$. Suppose some $(i,s)$ is a nil-inversion of $T$ for $s \neq j$, i.e., suppose $s>i$, where $s$ lies strictly below or to the left of $i$ in $T$ and $s\leq \hess(k_i)$. Then $(i,s)$ is also a nil-inversion of $T'$, because the entry to the right of $i$ in $T'$ is the same as that in $T$. The result follows similarly for $(j,s)$ a nil-inversion of $T$. 

\smallskip
\noindent {\sf Case 1b ($i>j$ and $k_i>k_j$):}  In this case, the resulting configuration in $T'$ is 
\[
\begin{ytableau}
    j & k_j\\
    i & k_i
\end{ytableau} \ .
\]
Since the entries to the right of both $i$ and $j$ remain the same, $\hess$-strictness is preserved for this sub-diagram. 
Since $i>j$ in this case, $(i,j)$ cannot be a nil-inversion of the original tableau $T$. However, if $i \leq h(k_j)$, then $(j,i)$ is a nil-inversion in $T'$. Now suppose some $(i,s)$ is a nil-inversion in $T$ for $s \neq j$. Then, by the same reasoning as in case (1a) above, $(i,s)$ is a nil-inversion in $T'$. Similar reasoning shows that any nil-inversion $(j,s)$ in $T$ is also a nil-inversion in $T'$. 

\smallskip
\noindent {\sf Case 2a ($i<j$ and $k_i > k_j$)}: 
Now the resulting configuration in $T'$ is 
\[
\begin{ytableau}
    i & k_j\\
    j & k_i
\end{ytableau} \ .
\]
Since $T$ is $\hess$-strict, we know $i<j\leq \hess(k_j)$. Since $k_j < k_i$ and Hessenberg functions are non-decreasing, $\hess(k_j) \leq \hess(k_i)$. Putting these together we obtain $j \leq \hess(k_i)$. Thus this re-ordering preserves $\hess$-strictness in this sub-diagram.

We now consider the nil-inversions. We see the only nil-inversions of $T$ that may be affected are of the form $(i,j)$, $(j,s)$ where $s>j$, or $(i,s)$ where $s>i$. If $(i,j)$ is a nil-inversion in $T$, then $(i,j)$ is also a nil-inversion in $T'$, since $j\leq \hess(k_j)$ by $\hess$-strictness of $T$. 
Next, suppose $(j,s)$ is a nil-inversion with $j<s$. In particular, $s \leq \hess(k_j)$. Then $(j,s)$ is also a nil-inversion of $T'$ since $s\leq \hess(k_j)\leq \hess(k_i)$. 

Finally, suppose $(i,s)$ is a nil-inversion in $T$ where $i<s \leq \hess(k_i)$ and $s \neq j$. However, since $k_j<k_i$, it may happen that $s \not \leq \hess(k_j)$, i.e.,  $(i,s)$ is not a nil-inversion in $T'$. Note that if $s \not \leq \hess(k_j)$ then we must have $j<s$, since if $j>s$ then $s<j\leq \hess(k_j)$, contradicting the assumption. Thus, we have seen that if $(i,s)$ is a nil-inversion in $T$ but $(i,s)$ is not a nil-inversion in $T'$, then $s \leq \hess(k_i)$, $ s>\hess(k_j)$, and $j<s$. This implies that $(j,s)$ is a nil-inversion in $T'$. On the other hand, note that $(j,s)$ cannot be a nil-inversion in $T$, since $s>\hess(k_j)$. Thus, it follows that if we lose a pair $(i,s)$ (i.e., it is a nil-inversion in $T$ but not in $T'$) then we gain a pair $(j,s)$ (i.e., $(j,s)$ was not a nil-inversion in $T$ but it is a nil-inversion in $T'$).

\smallskip
\noindent {\sf Case 2b ($i>j$ and $k_i<k_j$)}: 
In this case, the resulting configuration in $T'$ is
\[
\begin{ytableau}
    j & k_i\\
    i & k_j
\end{ytableau} \ .
\]
We first claim that $\hess$-strictness is preserved. Indeed, by $\hess$-strictness of the original $T$, 
we know $j<i\leq \hess(k_i)$. Moreover, since Hessenberg functions are non-decreasing, we also have $i\leq \hess(k_i)\leq \hess(k_j)$. Thus this re-ordering preserves $\hess$-strictness.

We now consider the nil-inversions. Again, the only inversions that may be affected are of the form $(i,s)$ where $s>i$ or $(j,s)$ where $s>j$. Note that $(j,i)$ is not a nil-inversion of $T$ since $i$ was not below $j$ in the same column, but $(j,i)$ is a nil-inversion of $T'$ since $i$ is now below $j$ and $i\leq h(k_i)$ because $T$ is $h$-strict.

Now suppose $(i,s)$ is a nil-inversion of $T$ with $s>i$, $s \neq j$. Then $(i,s)$ is also a nil-inversion of $T'$ since, by assumption, $s\leq \hess(k_i)$, and since $k_i < k_j$ by assumption in this case, we also know $\hess(k_i) \leq \hess(k_j)$. This implies $s \leq \hess(k_j)$, as required. Lastly, suppose $(j,s)$ is a nil-inversion of $T$ where $s>j$ and $s \neq i$. Suppose further that $(j,s)$ is not a nil-inversion of $T'$, i.e., $s \not \leq \hess(k_i)$. We claim that this implies that $i<s$. To see this, suppose that $i >s$. Then $s<i\leq \hess(k_i)$, but this contradicts the above. Thus we conclude both that $i<s$ and (since $(j,s)$ is a nil-inversion in $T$) $s\leq \hess(k_j)$. This implies that $(i,s)$ is a nil-inversion of $T'$. Note that, from the fact that $s \not \leq \hess(k_i)$, we know that $(i,s)$ cannot be a nil-inversion of $T$. It follows that (similar to Case 1b above), if we lose a pair $(j,s)$ (i.e., it is a nil-inversion of $T$ but it is not a nil-inversion of $T'$) then we gain a pair $(i,s)$ (i.e., $(i,s)$ is not a nil-inversion of $T$ but it is a nil-inversion of $T'$). 

From the above, it follows that the number of nil-inversions in $T'$ is greater than the number of nil-inversions in $T$ in this case because we have gained at least one nil-inversion, namely $(j,i)$.  

\smallskip

We now discuss Case 3, where there is a box to the right of $i$ but none to the right of $j$ in $T$, and Case 4, where there is no box to the right of both $i$ and $j$ in $T$. In fact, the arguments in these cases follow by essentially the same reasoning as Cases 1a and 2b, respectively, as the absence of the boxes in these cases imply that there are fewer conditions to check. As in the Case 2b above, if $i>j$ in either of Case 3 or Case 4 and hence the resulting configuration in $T'$ places the $i$ above the $j$, then $T'$ gains at least one nil-inversion. 

Finally, note that one of Case 2b, or the Case 3 with $i>j$, or Case 4 with $i>j$, always occurs at least once in the process of sorting rows $r$ and $r+1$. Indeed, one of these cases occurs when sorting the pair $(i,j)$ with $i>j$ such that $i$ is in row $r$, $j$ is in row $r+1$, and so that this pair occurs in column $c$ with $c$ as large as possible having an inversion between entries in the rows. But as we saw above, in each of these cases, $T'$ gains at least one nil-inversion. In all other cases, the proof above shows that the number of nil-inversions in $T'$ is weakly greater than the number in $T$. Since the strict increase must occur at least once, this proves (ii). We have also seen above that $\hess$-strictness is preserved in every case. This shows (i), so the claim follows.  
\end{proof}

\begin{example}\label{ex:nilpSort}
Consider $\hess=(3,3,3,4,6,6,7)$ and $\tabExtra{w,\lambda}$ below where $\lambda=(3,3,1)$. To the right, we consider the tableau $\tabExtra{\tilde{w},\lambda}$ formed by sorting the columns of $\tabExtra{w,\lambda}$ into increasing order.
    \[\tabExtra{w,\lambda}=
\begin{ytableau}
    3 & 2 & 7\\
    1 & 4 & 5\\
    6
\end{ytableau}
\quad
\rightarrow
\quad
\tabExtra{\tilde{w},\lambda}=
\begin{ytableau}
    1 & 2 & 5\\
    3 & 4 & 7\\
    6
\end{ytableau}\]
By Lemma~\ref{lem.inc-cols.nilp}, since $\tabExtra{{w},\lambda}$ is $\hess$-strict so is $\tabExtra{\tilde{w},\lambda}$. 
Further, we find:
\begin{align*}
    &\text{nil-inversions of $\tabExtra{w,\lambda}$: }(2,3),(2,4),(2,6),(4,6),(5,6)\\
    &\text{nil-inversions of $\tabExtra{\tilde w,\lambda}$: }(1,3),(2,3),(2,4),(2,6),(4,6),(5,6),(5,7).
\end{align*}
Thus by using Proposition~\ref{prop.nildim}, we see $\dim C_w\cap \Hess(\nilp_\lambda,\hess)\leq \dim C_{\tilde{w}}\cap \Hess(\nilp_\lambda,\hess)$, which confirms the statement of Lemma~\ref{lem.inc-cols.nilp} in this example.
\end{example}

Lemma~\ref{lem.inc-cols.nilp} completes the proof of Proposition~\ref{prop.steps}(2), showing the maximum dimension of the affine cells of $\Hess(\nilp_\lambda,\hess)$ must be achieved for $w$ such that $w\in \cm_{\lambda, \hess}$. 

It now remains to complete the proof of Proposition~\ref{prop.steps}(3), for which we began the argument in Lemma~\ref{lem.inc-cols.ss}. Indeed, in Lemma~\ref{lem.inc-cols.ss} we showed that sorting entries within columns to be increasing can only weakly increase the dimension of the corresponding affine cell in the semisimple Hessenberg variety. Given an arbitrary permutation $w$ and corresponding tableau $\tabExtra{w,\lambda}$ which is not $\hess$-strict, it remains to find a sequence of permutations (or equivalently, their corresponding tableau) which is also non-decreasing in dimensions of corresponding affine cells, and terminates at an $\hess$-strict tableau. This task occupies the remainder of this section.  We start with Definitions~\ref{def.non-h-strict-cons-pairs}
and~\ref{def.swap} below. 

\begin{definition}\label{def.non-h-strict-cons-pairs} Let $\tab$ be a column increasing tableau of shape $\lambda$. Assume that $\tab$ is \emph{not} $\hess$-strict, i.e., there exists at least one pair $(i,j)$ appearing in consecutive boxes in $\tab$, such as $\begin{ytableau}i & j \end{ytableau}$\,,  for which 
\begin{eqnarray}\label{eqn.h-violation}
i> \hess(j).
\end{eqnarray}
Set
\[\imax :=\max\{i\in[n]  :  (i,j) \text{ appears as $\begin{ytableau}i & j \end{ytableau}$ in $T$ and satisfies} ~\eqref{eqn.h-violation}\}.\]
We call $\imax$ the \emph{maximal non-$\hess$-strict label} in $\tab$.
\end{definition} 

For the purpose of this discussion, if a pair $(i,j)$ appears in horizontally consecutive boxes in $\tab$ as $\begin{ytableau}i & j \end{ytableau}$ then we call it a \emph{consecutive pair} (in $\tab$). 
Now we outline a procedure to produce a new tableau from a tableau $\tab$, where $\tab$ is not $\hess$-strict and has increasing columns (as in Definition~\ref{def.non-h-strict-cons-pairs}). Suppose that $\imax$ in $T$ appears in column $k$ of $\tab$. Then we know there exists at least one consecutive pair  in columns $k$ and $k+1$ satisfying~\eqref{eqn.h-violation}. Now we define the following notation.  
\begin{itemize}
\item Reading from top to bottom, let $(i_0,j_0)$ be the \emph{first} consecutive pair in columns $k$ and $k+1$ satisfying~\eqref{eqn.h-violation}, and 
\item reading from top to bottom, let $(i_1, j_1)$ be the \emph{last} consecutive pair in columns $k$ and $k+1$ satisfying~\eqref{eqn.h-violation}.
\end{itemize}
We note that it is possible for $(i_0,j_0)=(i_1,j_1)$. Since $T$ is column increasing, $i_1=\imax$.

\begin{definition}\label{def.swap} Following the notation of Definition~\ref{def.non-h-strict-cons-pairs} and the discussion above, we define: 
\begin{itemize} 
\item $\tilde{\tab}(k)$ to be the tableau obtained from $\tab$ by exchanging the entries $j_0$ and $i_1$, and then re-sorting the columns in increasing order (reading top to bottom), and 
\item Since $k$ is well-defined from $T$ as the column containing~$\imax$, we use the notation $\swap{\tab}:=\tilde{\tab}(k)$.
\end{itemize}
\end{definition}

\begin{rem}
    If $(i_0,j_0)=(i_1,j_1)$ in the construction above, then $\tilde{\tab}(k)$ is the tableau obtained from $\tab$ by exchanging entries $j_0$ and $i_0$ and then re-sorting the columns to be in increasing order. 
\end{rem}

\begin{example}\label{ex:swapTab}
Consider $\hess=(2,2,4,4,5,7,7,9,11,12,12,12)$ and $\tabExtra{w,\lambda}$ below. Here, $T$ is not $\hess$-strict and $\imax=10$, which lies in the second column. Thus $\swap{\tabExtra{w,\lambda}}=\tilde{\tabExtra{w,\lambda}}(2)$. 
In $\tabExtra{w,\lambda}$, we have shaded the boxes containing $(i_0,j_0)$ in light gray and the boxes containing $(i_1,j_1)$ in darker gray. The middle diagram 
is $\tabExtra{w,\lambda}$ with $j_0$ and $i_1$ exchanged. The rightmost diagram is $\swap{\tabExtra{w,\lambda}}$, formed by sorting the columns of the middle tableau into increasing order.
\[\tabExtra{{w},\lambda}=
\begin{ytableau}
    2 &  *(lightgray) 3 & *(lightgray) 1 \\
    5 & 7 & 4 \\
    6 & *(gray) 10 & *(gray) 8 \\
    12 & 11 & 9
\end{ytableau}
\quad
\longrightarrow
\quad
\begin{ytableau}
    2 &   3 &  10 \\
    5 & 7 & 4 \\
    6 &  1 & 8 \\
    12 & 11 & 9
\end{ytableau}
\quad
\longrightarrow
\quad
\begin{ytableau}
    2 &   1 &  4 \\
    5 & 3 & 8 \\
    6  &  7 & 9 \\
    12 & 11 & 10
\end{ytableau} \ .
\]
\end{example}

Now we show in Lemma~\ref{lemma.SScells-swap} that when we perform the transformation  $\tab \mapsto \swap{T}$ of Definition~\ref{def.swap}, the dimension of the corresponding Hessenberg--Schubert cells in the semisimple Hessenberg variety $\Hess(S_{\lambda^t},\hess)$ is non-decreasing.

\begin{lemma}\label{lemma.SScells-swap} Let $\tabExtra{w,\lambda}$ be a column increasing tableau that is not $\hess$-strict with  $\imax$ the maximal non-$\hess$-strict label. Let $\tilde{w}\in S_n$ be the permutation such that $\swap{\tabExtra{w,\lambda}} = \tabExtra{\tilde{w},\lambda}$. Then 
\[
\dim \left( C_w\cap \Hess(S_{\lambda^t},\hess) \right) \leq \dim \left( C_{\tilde{w}}\cap \Hess(S_{\lambda^t},\hess)\right).
\]
\end{lemma}

\begin{proof}
Consider the coset decomposition $w=yv$ for $y\in W_{\lambda^t}$ and $v\in {^{\lambda^t}W}$. 
Similarly, we have $\tilde{w}=\tilde{y}\tilde{v}$ for $\tilde{y}\in W_{\lambda^t}$ and $\tilde{v}\in {^{\lambda^t} W}$.

We first claim that $\ell(y)= \ell(\tilde{y})$. To see this, recall from Lemma~\ref{lemma.tabSS cardinality}(1) that $\ell(y)$ is the number of inversions occurring within the columns of $T$. Since both $\tabExtra{w,\lambda}$ and $\swap{\tabExtra{w,\lambda}}$ are defined to be column increasing, we see that $\ell(y)= \ell(\tilde{y})$. (In fact, $y=\tilde y$ is the longest element of the parabolic subgroup $W_{\lambda^t}$.) 

By Proposition~\ref{prop.ssdim}, it suffices to show that  $\ell_\hess(v)\leq \ell_\hess(\tilde{v})$.
By Lemma~\ref{lemma.tabSS cardinality}  
we can compute these quantities by examining the tableaux $\tabExtra{w,\lambda}$ and $\tabExtra{\tilde{w},\lambda}$. Following Definition~\ref{def.swap}, suppose $\imax$ occurs in column $k$ and $k+1$. We illustrate these two columns of $\tabExtra{w,\lambda}$ below, as well as the initial swap that is part of the construction of $\swap{\tabExtra{w,\lambda}} = \tabExtra{\tilde{w},\lambda}$. 
\[\begin{ytableau}
     \vdots & \vdots        \\
    i_0 & j_0\\
    \vdots & \vdots        \\
    i_1 & j_1\\
     \vdots & \vdots        
\end{ytableau}\longrightarrow
\begin{ytableau}
 \vdots & \vdots        \\
    i_0 & i_1\\
    \vdots & \vdots        \\
    j_0 & j_1\\
     \vdots & \vdots        
\end{ytableau}\]
To obtain $\tabExtra{\tilde{w},\lambda}$ we sort the columns of the resulting tableau into increasing order (reading top to bottom). 

By the hypothesis, the columns of $\tabExtra{{w},\lambda}$ are increasing, so we know $i_1>i_0$. By assumption, $i_0>\hess(j_0)$. Because the Hessenberg function $\hess$ is weakly increasing, this implies that $i_0>j_0$, so $i_1>j_0$ also. 
Additionally, this gives that $i_1>i_0>\hess(j_0)$.

Recall from Lemma~\ref{lemma.tabSS cardinality}(2) that $\ell_\hess(v)$ and $\ell_\hess(\tilde{v})$ count the pairs $(i,j)$ with $i<j\leq \hess(i)$ in the respective tableaux such that $j$ appears in any column strictly to the left of the column containing $i$. In order to prove $\ell_\hess(v)\leq \ell_\hess(\tilde{v})$, we will describe an injection $\phi:\hinv{\hess,v}\rightarrow\hinv{\hess,\tilde{v}}$.
Let $(a<b)\in \hinv{\hess,v}$. To define $\phi$ on this pair, we take cases. 

\smallskip

\noindent{\sf Case 1 ($(a<b)=(j_0<i_1)$):}
In order for $(j_0,i_1)$ to appear in $\hinv{\hess,v}$, we must have $i_1 \leq \hess(j_0)$. However, the discussion above showed $i_1 > \hess(j_0)$, so this case never occurs.

\smallskip
\noindent{\sf Case 2 ($\{a,b\}\cap \{i_1,j_0\}=\emptyset$):} By the definition of $\tabExtra{\tilde{w},\lambda}$, it follows that if $(a<b) \in \hinv{\hess,v}$ then $(a<b)\in \hinv{\hess,\tilde{v}}$. We define $ \phi(a<b) :=(a<b)$ in this case.

\smallskip
\noindent{\sf Case 3 ($\{a,b\} \cap \{i_1,j_0\} = \{j_0\}$):} In this case, either $a=j_0$ and $b \neq i_1$ or $a \neq i_1$ and $b=j_0$.  Note that if $b=j_0$ then since $(a<b) \in \hinv{\hess,v}$ by assumption, $a$ must lie in a column strictly to the right of column $k+1$. Similarly, if $a=j_0$, then $b$ must occur in column $k$ or in an earlier column of $\tabExtra{w,\lambda}$.  It follows that if either $b=j_0$ or $a=j_0$ and $b$ lies strictly to the left of column $k$, then we have that $(a<b) \in \hinv{\hess,\tilde{v}}$. Thus, in these cases, we define $\phi(a<b) := (a<b)$. 

The remaining case is when $a=j_0$ and $b$ lies in column $k$. Since $(a=j_0<b)$ is in $\hinv{\hess,w}$, we know $b\leq \hess(j_0)$. Then since $\hess(j_0)<i_0$ by assumption on $(i_0,j_0)$, it follows that $b<i_0$. Since the columns of $\tabExtra{w,\lambda}$ are increasing by assumption, $b$ must appear above $i_0$ in the $k$-th column. Let $c$ denote the entry in the $(k+1)$-th column directly to the right of $b$ in $\tabExtra{w,\lambda}$. We can visualize columns $k,k+1$ in $\tabExtra{{w},\lambda}$ as well as $\swap{\tabExtra{w,\lambda}}$ as: 
\[\begin{ytableau}
     \vdots & \vdots        \\
     b & c        \\
     \vdots & \vdots        \\
    i_0 & j_0\\
    \vdots & \vdots        \\
    i_1 & j_1\\
     \vdots & \vdots        
\end{ytableau}
 \longrightarrow
 \begin{ytableau}
  \vdots & \vdots        \\
      b & c        \\
  \vdots & \vdots        \\
     i_0 & i_1\\
     \vdots & \vdots        \\
     j_0 & j_1\\
      \vdots & \vdots        
 \end{ytableau} \ .
\]
By the construction of $\swap{\tabExtra{w,\lambda}}$, the pair $(i_0,j_0)$ is the first pair (reading from top to bottom) satisfying~\eqref{eqn.h-violation}, so it follows that $b\leq \hess(c)$. We see $j_0=a<b\leq \hess(c)$, so $(c<j_0)\in\hinv{\hess,\tilde{v}}$. We define $ \phi(j_0<b) :=(c<j_0)$ in this case.

\smallskip

\noindent{\sf Case 4 ($\{a,b\} \cap \{i_1,j_0\} = \{i_1\}$):}  In this case, either $a=i_1$ and $b \neq j_0$ or $a \neq j_0$ and $b=i_1$. By reasoning similar to Case 3, if $a=i_1$ then $b$ has to lie in a column strictly to the left of column $k$ of $\tabExtra{w,\lambda}$. If $b=i_1$ then $a$ lies in column $k+1$ or strictly to the right of column $k+1$. If $a=i_1$ or if $b=i_1$ and $a$ lies strictly to the right of column $k+1$, then $a<b$ remains an element in $\hinv{\hess,\tilde{v}}$ and we define $\phi(a<b) = (a<b)$ in these cases. 

The remaining case is when $a<b=i_1$ and $a$ lies in column $k+1$. 
By the fact that $(a<i_1)\in\hinv{\hess,v}$, we know $i_1\leq \hess(a)$. Then since Hessenberg functions are increasing and $\hess(j_1)<i_1$ by assumption, this forces $j_1<a$. By reasoning as in Case 3, it follows that $a$ lies below $j_1$ in column $k+1$ of $\tabExtra{w,\lambda}$. Let $d$ denote the entry of the box directly to the left of $a$ in column $k$. We can visualize columns $k,k+1$ in $\tabExtra{{w},\lambda}$ and  $\swap{\tabExtra{w,\lambda}}$ as: 
\[\begin{ytableau}
     \vdots & \vdots        \\
    i_0 & j_0\\
    \vdots & \vdots        \\
    i_1 & j_1\\
      \vdots & \vdots        \\
     d & a        \\
     \vdots & \vdots  
\end{ytableau}
 \longrightarrow
 \begin{ytableau}
     \vdots & \vdots        \\
    i_0 & i_1\\
    \vdots & \vdots        \\
    j_0 & j_1\\
      \vdots & \vdots        \\
     d & a        \\
     \vdots & \vdots  
\end{ytableau} \ .
\]
Similarly to Case 3, the pair $(i_1,j_1)$ is the last consecutive pair satisfying~\eqref{eqn.h-violation}, so $d\leq \hess(a)$. Now since $a<i_1$ we know $\hess(a)\leq \hess(i_1)$. Thus $d\leq \hess(i_1)$, so $(i_1<d)\in\hinv{\hess,\tilde{v}}$.
We define $ \phi(a<i_1) :=(i_1<d)$ in this case.

Considering the above cases together, it is straightforward to check that $\phi$ is indeed an injection. Thus we have $\ell_\hess(v)\leq \ell_\hess(\tilde{v})$, and the result follows. 
\end{proof}

\begin{example}\label{ex:swapTabDim}
Continuing Example~\ref{ex:swapTab}, let $\swap{\tabExtra{w,\lambda}} = \tabExtra{\tilde{w}, \lambda}$. 
Let $w=yv$ and $\tilde{w}=\tilde{y}\tilde{v}$ for $y,\tilde{y}\in W_{\lambda^t}$ and $v,\tilde{v}\in {^{\lambda^t}W}$. By column increasingness we see $\ell(y)=\ell(\tilde{y})=3\cdot \binom{4}{2}$. Additionally, we see  $5=\ell_\hess(v)\leq \ell_\hess(\tilde{v})=5$, so
$8=\dim \left( C_w\cap \Hess(S_{\lambda^t},\hess)\right) \leq \dim \left( C_{\tilde{w}}\cap \Hess(S_{\lambda^t},\hess)\right)=8$ in this case, as indicated by Lemma~\ref{lemma.SScells-swap}.
\end{example}

We now make use of the ``swap'' transformation to build a sequence of tableau. The goal is to construct a tableau which is both $\hess$-strict and is column increasing. We must show that this is possible, and that the process ends in finitely many steps. This is the content of the next lemma.

\begin{lemma}\label{lem:swapCol} Let $\tab$ be a tableau of shape $\lambda$ that is column increasing but is not $\hess$-strict. Consider the sequence $\tab_1, \tab_2, \ldots$ of column increasing tableau of shape $\lambda$ such that $\tab_1= \tab$ and $\tab_i = \swap{\tab_{i-1}}$ for $i>1$. Then there exists some $\ell\geq 2$ such that $\tab_\ell$ is $\hess$-strict, i.e., this sequence terminates.
\end{lemma}

We first illustrate the lemma in a worked example. 

\begin{example}\label{ex:nilpSwapSeq}
Consider $\hess=(2,3,3,4,6,6,7,8)$. Below is a sequence $T_1,T_2,T_3,T_4,T_5$, where $\tab_i = \swap{\tab_{i-1}}$, terminating in an $\hess$-strict tableau $T_5$. Each arrow between $\tab_i$ and $\tab_{i+1}$ is labeled with the maximal non-$\hess$-strict label $\imax$ in $\tab_i$, and in each tableau, the consecutive pairs $(i_0,j_0)$ and $(\imax = i_1,j_1)$ are highlighted.
    \[\tab_1=
\begin{ytableau}
   *(lightgray) 6 & *(lightgray) 3 & 1\\
    7 & 4 & 2\\
    *(lightgray) 8 & *(lightgray) 5
\end{ytableau}
\ \ 
\xrightarrow{\,8\,}
\ \ 
\begin{ytableau}
    3 & *(lightgray) 4 & *(lightgray) 1\\
    6 & *(lightgray) 5 & *(lightgray) 2\\
    7 & 8
\end{ytableau}
\ \ 
\xrightarrow{\,5\,}
\ \ 
\begin{ytableau}
    *(lightgray) 3 & *(lightgray) 1 & 2\\
    *(lightgray) 6 &  *(lightgray) 4 & 5\\
    7 & 8
\end{ytableau}
\ \ 
\xrightarrow{\,6\,}
\ \ 
\begin{ytableau}
    1 & *(lightgray) 4 & *(lightgray) 2\\
    3 & 6 & 5\\
    7 & 8
\end{ytableau}
\ \ 
\xrightarrow{\,4\,}
\ \ 
\begin{ytableau}
    1 & 2 & 4\\
    3 & 6 & 5\\
    7 & 8
\end{ytableau}\,.\]
Lemma~\ref{lem:swapCol} ensures that this sequence will terminate with an $\hess$-strict tableau.
\end{example}

The proof of Lemma~\ref{lem:swapCol} requires the following preliminary result.

\begin{lemma}\label{prop:lenSwaps}
Suppose $\tabExtra{{w},\lambda}$ is a  column increasing tableau  such that $\tabExtra{{w},\lambda}$ is not $\hess$-strict. Let $\tabExtra{\tilde{w},\lambda}=\swap{\tabExtra{w,\lambda}}$. Then $\ell(w)> \ell(\tilde{w})$.
\end{lemma}

\begin{proof}
Suppose $i_1=\imax$ is the maximal non-$\hess$-strict label in $\tabExtra{w,\lambda}$, lying in column $k$. Let $(i_0,j_0)$ and $(i_1,j_1)$ be the first and the last consecutive pairs in columns $k$ and $k+1$ satisfying~\eqref{eqn.h-violation}, respectively. Since $T$ is column increasing, $i_1\geq i_0$ (with equality occurring when $(i_0,j_0)=(i_1, j_1)$). Since $(i_0,j_0)$  satisfies~\eqref{eqn.h-violation}, by the definition of Hessenberg functions, we know also $i_0>j_0$ and thus $i_1>j_0$. Now $\swap{\tabExtra{w,\lambda}}= \tabExtra{\tilde{w},\lambda}$ is the tableau obtained by exchanging $j_0$ and $i_1$ and re-sorting the columns. 

Consider the coset decompositions $w=yv$ and $\tilde{w}=\tilde{y}\tilde{v}$ with $y,\tilde{y}\in W_{\lambda^t}$ and $v,\tilde{v}\in {^{\lambda^t}}W$. Since both $\tabExtra{w,\lambda}$ and $\tabExtra{\tilde w,\lambda}$ are column increasing, $y=\tilde{y}$ is the longest element of $W_{\lambda^t}$. Using Lemma~\ref{lem.cosets}, the desired statement follows once we prove $\ell(v)> \ell(\tilde v)$. 

Since $i_1>j_0$ and $i_1$ is in column $k$ while $j_0$ is in column $k+1$, we have by Lemma~\ref{lem.cosets.tab}(2) that the pair $(j_0, i_1)$ is an inversion of $v$, but $(j_0, i_1)$ is not an inversion of $\tilde{v}$ after the swap. Since all other entries in the columns of the respective tableaux remain the same (up to reordering), we have $\ell(v)=\ell(\tilde{v})+1$ and thus $\ell(v)>\ell(\tilde v)$.
\end{proof}

\begin{proof}[Proof of Lemma~\ref{lem:swapCol}]
Let $S=(w^{(i)})_{i\geq 1}$ be the sequence of permutations where $\tabExtra{w^{(i)},\lambda}=\tab_i$.  By Lemma~\ref{prop:lenSwaps},  we see the sequence whose $i$-th entry is $\ell(w^{(i)})$ is a strictly decreasing sequence of nonnegative integers. Thus $S$ must be finite. 
\end{proof}

Figure~\ref{fig:tree} illustrates Lemma~\ref{lem:swapCol} across multiple steps. It shows how many, potentially different, tableaux map to the same element of $\cm_{\lambda, \hess}$. In the figure, we highlight the pairs $(i_0, j_0)$ and $(i_1, j_1)$ to indicate the entries swapped in each application of $\swap{-}$. 

With Lemma~\ref{lem:swapCol} in hand, we have completed the proof of all parts of Proposition~\ref{prop.steps}. 
Now we are ready to prove the main theorem of this section. 

\begin{proof}[Proof of Theorem~\ref{thm.ss-nilp}]
The statement that $\dim \left( \Hess(\nilp_\lambda, \hess)\right) = \dim \left( \Hess(\semi_{\lambda^t}, \hess)\right)$ follows immediately from Proposition~\ref{prop.steps}. Recall that Lemma~\ref{lemma: orbits in g lambda} describes the semisimple and nilpotent orbits in the sheet $\fg_\lambda$. In particular, we have $\nilp_\lambda, \semi_{\lambda^t}\in \fg_\lambda$. 

Given any nilpotent matrix $\nilp$ in $\fg_\lambda$, we know $\nilp$ must have Jordan type $\lambda$ and thus $\nilp$ and $\nilp_\lambda$ are conjugate. Any semisimple matrix $\semi\in \fg_\lambda$ must be conjugate to $\semi_{\lambda^t}$ where $\semi_{\lambda^t}$ is as in Definition~\ref{definition.semi.tau} and has the same eigenvalues as $\semi$. The last statement of the theorem now follows immediately from the fact that $\dim \left( \Hess(\nilp_\lambda, \hess)\right) = \dim \left( \Hess(\semi_{\lambda^t}, \hess)\right)$ by Lemma~\ref{lemma.conjugation}. 
\end{proof}

\begin{figure}
\begin{tikzpicture}[
    scale=1,
    every node/.style={font=\sffamily},
    box/.style={draw, thick, minimum width=0.8cm, minimum height=0.6cm, inner sep=2pt},
    tri/.style={draw, thick, regular polygon, regular polygon sides=3, inner sep=1pt},
    blacknode/.style={fill, draw, minimum size=3mm, inner sep=0pt}
]
 \node (C) at (0,0) {$\ytableausetup{boxsize=1.1em}
{\begin{ytableau}
*(lightgray)6 & *(lightgray) 2 & 1 \\
7 & 3 & 5 \\
*(lightgray) 8 & *(lightgray)4
\end{ytableau}}$};

\node (B2) at (3,0) {$\ytableausetup{boxsize=1.1em}
{\begin{ytableau}
2 & 3 & 1 \\
*(lightgray) 6 & *(lightgray) 4 & 5 \\
7 & 8
\end{ytableau}}$};

\node (B1) at (0,-2) {$\ytableausetup{boxsize=1.1em}
{\begin{ytableau}
2 & 3 & 1 \\
4 & 6 & 5 \\
*(lightgray) 8 & *(lightgray) 7
\end{ytableau}}$};

\node (B3) at (6,-2) {$\ytableausetup{boxsize=1.1em}
{\begin{ytableau}
2 & 3 & 1 \\
*(lightgray)7 & *(lightgray) 4 & 5 \\
*(lightgray) 8 & *(lightgray)6
\end{ytableau}}$};

\node (B5) at (0,-4) {$\ytableausetup{boxsize=1.1em}
{\begin{ytableau}
*(lightgray)4 & *(lightgray) 2 & 1 \\
7 & 3 & 5 \\
*(lightgray) 8 & *(lightgray)6
\end{ytableau}}$};

\node (A1) at (3,-4) {$\ytableausetup{boxsize=1.1em}
{\begin{ytableau}
2 & *(lightgray)3 & *(lightgray) 1 \\
4 & *(lightgray) 6 &*(lightgray) 5 \\
7 & 8
\end{ytableau}}$};

\node (A2) at (9,-4) {$\ytableausetup{boxsize=1.1em}
{\begin{ytableau}
2 & 1 & 5 \\
4 & 3 & 6 \\
*(lightgray) 8 & *(lightgray) 7
\end{ytableau}}$};

\node (T) at (6,-4) {$\ytableausetup{boxsize=1.1em}
{\begin{ytableau}
2 & 1 & 5 \\
4 & 3 & 6 \\
7 & 8
\end{ytableau}}$};

    \draw[->] (C) -- (B2);
    \draw[->] (B2) -- (A1);
    \draw[->] (B1) -- (A1);
    \draw[->] (B3) -- (A1);
    \draw[->] (B5) -- (A1);
     \draw[->] (A1) -- (T);
     \draw[->] (A2) -- (T);
\end{tikzpicture}
     \caption{For $\lambda=(3,3,2)$ and $h=(2,3,4,4,5,7,7,8)$, above is a tree of tableaux $\tab_{w,\lambda}$ where the arrows connect tableaux $\tab_{w,\lambda}\rightarrow\swap{\tab_{w,\lambda}}$. In each tableau, the consecutive pairs $(i_0,j_0)$ and $(i_1,j_1)$ used for the swap are shaded in gray. The tableaux $\tab_{w,\lambda}$ in the first row all have $\dim \left(C_w\cap \Hess(S_{\lambda^t},\hess) \right)=8$, and those in the second row are $9$-dimensional, and those in the bottom row are $10$-dimensional.} 
     \label{fig:tree}
\end{figure}
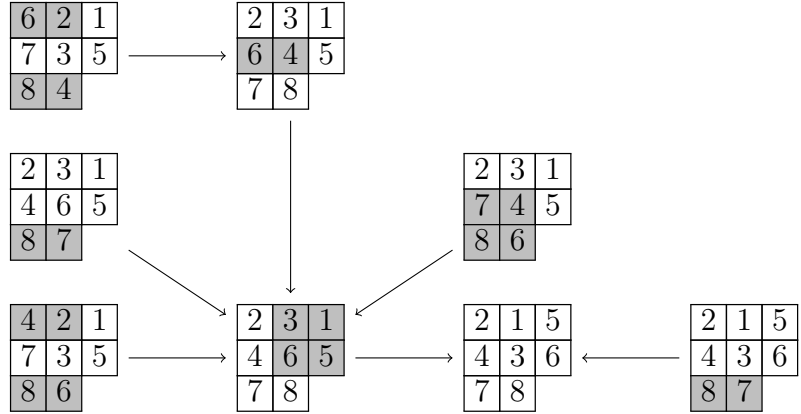


\section{All Hessenberg varieties over a fixed sheet $\fg_\lambda$}\label{sec.general case}

In this section we extend the dimension stability result of Theorem~\ref{thm.ss-nilp} to all Hessenberg varieties $\Hess(\mx,\hess)$ for $\hess:[n]\to [n]$ as $\mx$ varies over the sheet $\fg_\lambda$. The main result of this section is the following. 

\begin{thm} \label{thm.main.general}
Let $\lambda\vdash n$ and $\hess: [n] \to [n]$ be a Hessenberg function. Then the dimension $\dim(\Hess(\mx,\hess))$ is constant as $\mx$ varies over the sheet $\fg_\lambda$, i.e., for any $\mx,\mx' \in \fg_\lambda$, we have $\dim(\Hess(\mx,\hess)) = \dim(\Hess(\mx',\hess))$. 
\end{thm}

As in the previous section, the proof requires several steps. In our arguments below, we use an inductive formula of the second author (Proposition~\ref{prop.gendim} below) for the dimension of a Hessenberg--Schubert cell. We briefly recall this context. 

Given a composition $\gm=(\gm_1,\cdots,\gm_k)\vDash n$, let $\semi_\mu$ denote a diagonal matrix as specified in Definition~\ref{definition.semi.tau}. Then we define 
\[
L_\gm := Z_G(\semi_\gm) := \{ g \in GL_n(\C)  :  g\semi_\gm g^{-1} = \semi_\mu \}. 
\]
This is a standard Levi subgroup of $GL_n(\C)$ isomorphic to $GL_{\mu_1}(\C) \times GL_{\mu_2}(\C) \times \cdots \times GL_{\mu_k}(\C)$; the matrices in $L_\gm$ are the block-diagonal $n\times n$ matrices with block sizes $\mu_1, \mu_2, \ldots, \mu_k$. Let $B_\mu:= B\cap L_\mu$ denote the Borel subgroup of $L_\mu$ obtained by intersecting with the subgroup of upper triangular matrices. 

The Lie algebra of $L_\mu$ is the Levi subalgebra of $\mathfrak{gl}_n(\C)$ defined by
\[
\fl_\mu := \fz(\semi_\mu) = \{X\in \mathfrak{gl}_n(\C): [X, \semi_\mu] = 0 \},
\]
where $\fz(\semi_\gm)$ denotes the centralizer of $\semi_\gm$. We have an isomorphism $\fl_\mu \cong \mathfrak{gl}_{\gm_1}(\C) \times \cdots \times \mathfrak{gl}_{\gm_k}(\C)$. Let $\mathfrak{b}_{\gm} := \mathfrak{b} \cap \fl_\gm$ denote the Lie algebra of $B_\mu$. 

For what follows, it is useful to introduce some terminology which allows us to define Hessenberg varieties using the identification of the flag variety $\Fl(\C^n)$ with the quotient space $GL_n(\C)/B$. To any Hessenberg function $\hess: [n] \to [n]$, we associate a \emph{Hessenberg subspace corresponding to $\hess$} as follows: 
\begin{equation}\label{eqn.Hessenberg.space}
H:= H(\hess) := \{ \mx \in \mathfrak{gl}_n(\C)  :  \mx_{ij} = 0 \, \textup{ for } \, i > \hess(j) \},
\end{equation}
where $\mx_{ij}$ denotes the $(i,j)$-th matrix entry of $\mx$. In general a subspace $H\subseteq \mathfrak{gl}_n(\C)$ is a \emph{Hessenberg space with respect to the Borel subalgebra $\fb$} (that is, $H=H(\hess)$ for some Hessenberg function $\hess$) if and only if $\fb+[\fb, H]\subseteq H$. We may generalize this in a straightforward manner to $\fl_\mu$ by taking products. 
Following the notation just established, we have the following:

\begin{lemma}[Precup~{\cite[Proposition 5.2]{Precup2013}}]\label{lem.cosetH} Let $\hess:[n] \to [n]$ be a Hessenberg function and $H$ its associated Hessenberg subspace. Let $\mu\vDash n$. For each $v\in  {^\mu}W$, the subspace $H_v:= \dot v H \dot v^{-1}\cap \fl_\mu$ of $\fl_\mu$ is a Hessenberg subspace of $\fl_\mu$ with respect to the Borel subalgebra $\fb_\mu:= \fb\cap \fl_\mu$. 
\end{lemma}

The Hessenberg space $H_v\subseteq \fl_\mu$ is uniquely determined by $\hess_v:=(\hess_v^{(1)}, \ldots, \hess_v^{(k)})$, where each $\hess_v^{(i)}:[\mu_i]\to [\mu_i]$ denotes the Hessenberg function corresponding to the Hessenberg space in $\mathfrak{gl}_{\mu_i}(\C)$ obtained from $H_v$ by projecting it to the diagonal $\mu_i\times \mu_i$ block in $\fl_\gm \cong \mathfrak{gl}_{\gm_1}(\C) \times \cdots \times \mathfrak{gl}_{\gm_k}(\C)$. Equivalently, we can view $\hess_v$ as a Hessenberg function on $[n]$ such that 
\begin{equation}\label{eqn.hess.v.on.i}
\hess_v(i)=i \, \text{ for all } \, i \in D_\gm
\end{equation}
and for each $j$ with $j\in \Lint{s}(\gm)$ for some $1\leq s \leq k = \len(\mu)$, 
\begin{equation}\label{eqn.hess.v.on.j}
\hess_v(j) = \mu_1+\cdots +\mu_{s-1} + \hess_v^{(s)}(j-\mu_1-\cdots - \mu_{s-1}). 
\end{equation}

\begin{example} If $\mu = (3,2)$, we have
\[
\fl_\mu = \begin{bmatrix} * & * & * & 0 & 0\\ * & * & * & 0 & 0\\ * & * & * & 0 & 0\\ 0 & 0 & 0 & * & *\\ 0 & 0 & 0 & * & *\\ \end{bmatrix}
\]
where $*$ indicates an arbitrary entry from $\C$.
Consider the Hessenberg space $H$ determined by the Hessenberg function $\hess=(2,4,4,5,5)$,
\[
H = \begin{bmatrix} * & * & * & * & *\\ * & * & * & * & *\\ 0 & * & * & * & *\\ 0 & * & * & * & *\\ 0 & 0 & 0 & * & *\\ \end{bmatrix}.
\]
If $v=[4\,1\,2\,5\,3]\in {^{(3,2)}}W$ then  
\[
H_v = \dot v H \dot v^{-1} \cap \fl_\mu = \begin{bmatrix} * & * & * & 0 & 0\\ * & * & * & 0 & 0\\ 0 & 0 & * & 0 & 0\\ 0 & 0 & 0 & * & *\\ 0 & 0 & 0 & 0 & *\\ \end{bmatrix}
\]
and thus $\hess_v = ((2,2,3),(1,2))$, or as a Hessenberg function on $[5]$ we have $\hess_v=(2,2,3,4,5)$.
\end{example}

We have just seen how a Hessenberg function $\hess$ and choice of $v \in {^\gm}W$ yields a sequence of Hessenberg functions $(h_v^{(1)}, \cdots ,h_v^{(k)})$ where $k =  \len(\mu)$. Given $\nilp \in \mathfrak{gl}_n(\C)$ a nilpotent element which lies in the Levi subalgebra $\fl_\gm$, we define $\bar{\nilp} := \nilp \vert_{\fl_\gm}$ to be the restriction of $\nilp$ to $\fl_\gm \cong \mathfrak{gl}_{\gm_1}(\C) \times \cdots \times \mathfrak{gl}_{\gm_k}(\C)$. Then the associated Hessenberg variety $\Hess(\bar{\nilp}, \hess_v)$ can be seen to satisfy an isomorphism 
\begin{equation}\label{eqn.nil.levi}
\begin{aligned}
\Hess(\bar{\nilp},\hess_v) \cong \Hess(\bar{\nilp}\vert_{\mathfrak{gl}_{\gm_1}(\C)}, \hess_v^{(1)})  \times \cdots & \times \Hess(\bar{\nilp} \vert_{\mathfrak{gl}_{\mu_k}(\C)}, \hess_v^{(k)}) \\ 
&\subset L_\gm/B_\gm \cong \prod_{i=1}^k \Fl(\C^{\gm_i}).
\end{aligned}
\end{equation}
In this sense, the Hessenberg variety $\Hess(\bar{\nilp},\hess_v)$ is a product of nilpotent Hessenberg varieties in smaller flag varieties. Moreover, for $z \in W_\mu \cong S_{\gm_1} \times S_{\gm_2} \times \cdots \times S_{\gm_k}$, we can write $z = z_1z_2\cdots z_k$ for $z_i \in S_{\mu_i}$ for $i \in [k]$, and we may consider $C_z := C_{z_1} \times \cdots \times C_{z_{k}}$ to be the product of the Schubert cells in each $\mathcal{F}\ell ags(\C^{\gm_i})$. Then we see that the Hessenberg--Schubert cell $C_z \cap \Hess(\bar{\nilp},\hess_v)$ is a product of Hessenberg--Schubert cells in the smaller flag varieties $\Fl(\C^{\gm_i})$. With this in mind, the next result gives an inductive formula for dimensions of Hessenberg--Schubert cells in terms of smaller Hessenberg--Schubert cells.

\begin{prop}[{\cite[Theorem 5.4, Corollary 5.5]{Precup2013}\label{prop.gendim}}] Let $\gm\vDash n$. Suppose $\mx\in \mathfrak{gl}_n(\C)$ has Jordan decomposition $\mx=\semi_\mu+\nilp$ where $\nilp$ is a nilpotent element of the Levi subalgebra $\fl_\mu := \mathfrak{z}(\semi_\mu)$. Let $\bar{\nilp}:= \nilp|_{\fl_\mu}$. Given $w\in S_n$, write $w=zu$ for $z\in W_\mu$ and $u\in  {^\mu W}$. Then the Hessenberg--Schubert cell $C_w\cap \Hess(\mx,\hess)$ is nonempty if and only if $C_z\cap \Hess (\bar\nilp, \hess_u)$ is nonempty. Moreover, if $C_w \cap \Hess(\mx,\hess)$ is nonempty, then it is isomorphic to affine space of dimension 
\[
\dim \left(C_w\cap \Hess (\mx,\hess)\right) =  \dim \left(C_z\cap \Hess (\bar\nilp, \hess_u)\right) + \ell_\hess(u).
\]
Moreover, the collection of nonempty Hessenberg--Schubert cells in $\Hess(\mx, \hess)$ is an affine paving of $\Hess(\mx, \hess)$.
\end{prop}

In what follows, we will use Proposition~\ref{prop.gendim} in order to show that any $\Hess(\mx, \hess)$ with $\mx$ in the sheet $\fg_\lambda$ has the same dimension as the semisimple Hessenberg variety in the same sheet. This is the content of Proposition~\ref{prop.gendim.equal.ssdim} below. 

We begin with some preliminaries. Let $\tau,\mu\vDash n$. We say that $\tau$ is a \emph{refinement of $\mu$} if $D_\gm\subseteq D_\tau$. It then follows from the definitions that $W_\tau$ is a subgroup of $W_\mu$ and ${^\mu}W \subseteq {^\tau}W$.

\begin{lemma} \label{lemma.refinement} Let $\tau,\mu\vDash n$ such that $\tau$ is a refinement of $\mu$. For $v\in {^\tau}W$, let $v=xu$ where $x\in W_\mu$ and $u\in {^\mu}W$. Then
\[
\ell_\hess(v) = \ell_\hess(u) + \ell_{\hess_u}(x).
\]
\end{lemma}

\begin{proof} We know from Lemma~\ref{lem.cosets} that the decomposition $v=xu$ is unique and length-additive. Hence from Lemma~\ref{lem.lengthadd} we may conclude 
\begin{equation}\label{eq: inv relation}
\inv(v) = \inv(u) \,  \sqcup \, u^{-1}(\inv(x)).
\end{equation}
Thus, $\hinv{\hess,u} \subseteq \hinv{\hess, v}$. Observe from~\eqref{eq: inv relation} that   
\begin{equation}\label{eq: complement}
\hinv{\hess, v}\setminus \hinv{\hess, u} = \{ (u^{-1}(i), u^{-1}(j)) : (i,j)\in \inv(x) \, \text{ and } \, u^{-1}(j)\leq \hess(u^{-1}(i)) \}.
\end{equation}
Thus, to complete the proof, it suffices to prove that there is a bijection between $\hinv{\hess_u,x}$, which counts $\ell_{\hess_u}(x)$, and the RHS of~\eqref{eq: complement}. 
Let $E_{a,b}$ denote the elementary matrix with unique nonzero entry equal to $1$ in position $(a,b)$.
Because $x\in W_\mu$, if $(i,j)\in \inv(x)$ then  
\[
\mu_1+\cdots + \mu_{s-1}+1\leq i <j \leq \mu_1 + \cdots + \mu_s
\] 
for some $s$. This implies, in particular, that $u^{-1}(i)<u^{-1}(j)$ because $u\in {^\mu}W$. We also have $E_{j,i} \in \fl_\mu$. Thus, for $(i,j)$ that satisfies $(i,j) \in \inv(x)$, we have
\begin{align*}
(i,j) \in \hinv{\hess_u,x} &\Leftrightarrow E_{j,i} \in H_u = \dot u H \dot u^{-1} \cap \fl_\mu \\
&\Leftrightarrow E_{j,i}\in \dot u H \dot u^{-1} \quad \text{(since $E_{j,i}\in \fl_\mu$)}  \\
&\Leftrightarrow E_{u^{-1}(j), u^{-1}(i)} \in H \\
&\Leftrightarrow u^{-1}(j) \leq \hess(u^{-1}(i)) \quad \text{(since $u^{-1}(i)<u^{-1}(j)$)}.
\end{align*}
Thus, $(i,j)\mapsto (u^{-1}(i), u^{-1}(j))$ is the required bijection $ \hinv{\hess_u, x} \to \hinv{\hess, v}\setminus \hinv{\hess, u}$.
\end{proof}

In what follows, we will consider matrices of a particular form. 

\begin{definition}\label{defn.lambda.compatible}
Let $\lambda$ be a partition of $n$. Let $\mx$ be an $n\times n$ matrix with $\mx \in \fg_\lambda$, so in particular, $\lambda_\mx = \lambda$. We say that $\mx$ is \emph{$\lambda$-standard} if $\mx = \semi_\gm +\nilp$ where: 
\begin{itemize}
\item $\semi_\mu$ is a diagonal matrix with distinct eigenvalues $c_1,c_2, \ldots, c_{k}$ where $k= \len(\mu)$ and the eigenvalue $c_i$ has algebraic multiplicity $\gm_i$, so that $L_\gm = Z_G(\semi_\mu) \cong GL_{\mu_1}(\C)\times GL_{\mu_2}(\C)\times \cdots \times GL_{\mu_k}(\C)$, 
\item the sizes of the Jordan blocks of $\mx$ corresponding to the eigenvalue $c_i$ determine a partition $\lambda^{(i)}$ of size $\mu_i = |\lambda^{(i)}|$ and $\lambda_i = \sum_{s=1}^k \lambda^{(s)}_i$, and 
\item $\nilp$ is nilpotent and $N \in L_\mu \cong GL_{\mu_1}(\C) \times \cdots \times GL_{\mu_k}(\C)$ and if $\bar\nilp_i$ denotes the projection of $\nilp$ to the $\mathfrak{gl}_{\gm_i}(\C)$ factor in $\fl_\gm \subset \mathfrak{gl}_n(\C)$, i.e., $\bar\nilp_i := \nilp \vert_{\mathfrak{gl}_{\gm_i}(\C)}$, then $\bar\nilp_i$ is equal to $\nilp_{\lambda^{(i)}}$, the nilpotent matrix associated to the base filling of the tableau of shape $\lambda^{(i)}$ as defined in Definition~\ref{def.base.nilp}. 
\end{itemize}
\end{definition}

Any matrix $\mx \in \fg_\lambda$ can be conjugated into $\lambda$-standard form by choosing an appropriate basis for each Jordan block.
We now define a new semisimple matrix, denoted $\semi_\tau$ from the data of a partition $\lambda$ and a $\lambda$-standard matrix.  

\begin{construction}\label{constr.semi.tau} Let $\lambda$ be a partition of $n$ and $\mx = \semi_\gm + \nilp$ a $\lambda$-standard matrix; we fix notation as in Definition~\ref{defn.lambda.compatible} above.
For each $i$ with $i\in[k]$, let $\tau^{(i)} = (\lambda^{(i)})^t$ be the transpose partition of $\lambda^{(i)}$. By construction, $\tau^{(i)}\vdash \mu_i$ is a partition of size $\mu_i$.   Let $\semi_{\tau^{(i)}}$ be a diagonal $\mu_i\times \mu_i$ matrix, with $\len(\tau^{(i)}) = \lambda^{(i)}_1$ distinct eigenvalues with multiplicities $\tau_1^{(i)},\cdots,\tau^{(i)}_{\len(\tau^{(i)})}$. We also assume that the eigenvalues of $\semi_{\tau^{(i)}}$ are pairwise distinct from the eigenvalues occurring in $\semi_{\tau^{(j)}}$ for all $j \neq i$. 

Now form $\tau\vDash n$ by concatenating the partitions $\tau^{(1)}, \tau^{(2)}, \ldots, \tau^{(k)}$. By construction, $\tau$ is a refinement of the composition $\mu$. Indeed, $\tau$ is obtained by replacing each part $\mu_i$ of $\mu$ with the partition $\tau^{(i)}$. 

Finally, we set $\semi_\tau$ equal to the $n\times n$ diagonal matrix such that the projection $\semi_\tau\vert_{\mathfrak{gl}_{\gm_i}(\C)}$ of $\semi_\tau$ onto the $\mathfrak{gl}_{\gm_i}(\C)$ factor in $\fl_\gm \subset \mathfrak{gl}_n(\C)$ is equal to $\semi_{\tau^{(i)}}$. 
\end{construction}

\begin{example}\label{example.semi.tau}
Let $n=9$ and let $\mx = \semi_\gm + \nilp$ where $\gm=(5,4)$ so $k=2=\len(\gm)$. Suppose $c_1=0$ and $c_2=1$ and $\lambda^{(1)} = (3,2)$ is a partition of $5 = \gm_1$, so there are $2$ Jordan blocks associated to the eigenvalue $c_1$, and $\lambda^{(2)} = (2,1,1)$ is a partition of $4=\gm_2$, so there are $3$ Jordan blocks associated to $c_2$. The base fillings of $\lambda^{(1)}$ and $\lambda^{(2)}$ are 
$$
\ytableausetup{centertableaux} \begin{ytableau}2 & 4 & 5\\ 1 & 3\\ \end{ytableau}
\qquad \qquad 
\begin{ytableau} 3 & 4 \\
2 \\
1 \\
\end{ytableau} 
$$
respectively, so 
$$
\bar\nilp_1 = \begin{bmatrix} 0 & 0 & 1 & 0 & 0 \\ 0 & 0 &0 & 1 & 0 \\
0 & 0 & 0 & 0 & 0 \\
0 & 0 & 0 & 0 & 1 \\
0 & 0 & 0 & 0 & 0 
\end{bmatrix}, \quad 
\bar\nilp_2 = \begin{bmatrix} 
0 & 0 & 0 & 0 \\
0 & 0 & 0 & 0 \\
0 & 0 & 0 & 1 \\
0 & 0 & 0 & 0 
\end{bmatrix} 
$$
and the semisimple matrix is 
$$
\semi_\gm = \textup{diag}(0,0,0,0,0,1,1,1,1).
$$
Notice that $\lambda_\mx  = (5,3,1)$ in this case.
We have $\tau^{(1)}=(\lambda^{(1)})^t = (2,2,1)$ and $\tau^{(2)} =(\lambda^{(2)})^t = (3,1)$. We choose $\semi_{\tau^{(1)}}$ to have $\lambda^{(1)}_1=3$ distinct eigenvalues of multiplicities $2,2,1$ respectively and $\semi_{\tau^{(2)}}$ to have $\lambda^{(2)}_1=2$ eigenvalues of multiplicities $3,1$ respectively; more concretely we may set 
$$
\semi_{\tau^{(1)}} = \diag(0,0,1,1,2) \;  \text{ and } \;
\semi_{\tau^{(2)}} = \diag(3,3,3,4)
$$
so $\semi_\tau = \diag(0,0,1,1,2,3,3,3,4)$. 
\end{example}

\begin{lemma}\label{lemma.sheet.same}
Following the notation and assumptions above, $\semi_\tau$ is in the same sheet as $\mx$, i.e., $\lambda_{\semi_\tau} =\lambda_\mx = \lambda$. 
\end{lemma}

\begin{proof} 
By definition of the matrix $\semi_\tau$, this follows once we show $\mathrm{sort}(\tau) = \lambda_\mx^t$ where $\mathrm{sort}(\tau)$ denote the partition obtained from the composition $\tau$ by reordering the parts to be weakly decreasing. Since the transpose operation on partitions is an involution, it suffices to show $\mathrm{sort}(\tau)^t = \lambda_{\mx}$. By definition,
\begin{align*}
\mathrm{sort}(\tau)_j^t &= \#\{p : \tau_p \geq j \}  = \sum_{i=1}^k \left( \#\{p: \tau_p^{(i)} \geq j\}\right) = \sum_{i=1}^k \lambda_j^{(i)} = \lambda_{X,j}.
\end{align*}
Thus, $\mathrm{sort}(\tau)^t=\lambda_{\mx}$ and we conclude $\semi_\tau \in \fg_{\lambda_\mx}$, as desired. 
\end{proof} 

\begin{example} 
Continuing Example~\ref{example.semi.tau} we see that $\semi_\tau$ has $5$ distinct eigenvalues, of multiplicities $2,2,1,3,1$, and since $\semi_\tau$ is semisimple, all Jordan blocks are size $1$. Thus the $5$ associated partitions are $(1,1),(1,1),(1),(1,1,1),(1)$ respectively, or 
\begin{equation}\label{eqn.refinement.tau}
\begin{ytableau}
 \\
 \\
\end{ytableau}
\quad
\begin{ytableau}
 \\
 \\
\end{ytableau}
\quad
\begin{ytableau}
 \\
\end{ytableau}
\quad \textup{ and } \quad 
\begin{ytableau}
 \\
 \\
 \\
\end{ytableau}
\quad
\begin{ytableau}
 \\
\end{ytableau}
\end{equation}
which is obtained by splitting $\lambda^{(1)} =(3,2)$ and $\lambda^{(2)} = (2,1,1)$ from Example~\ref{example.semi.tau} into separate columns. To compute the $j$-th entry of $\lambda_{\semi_\tau}$, we take the sum of the number of boxes in the $j$-th row of all the partitions. As we just observed, the tableau in~\eqref{eqn.refinement.tau} are just the columns appearing in the $\lambda^{(i)}$, we can see that $\lambda_\mx = \lambda_{\semi_\tau}=(5,3,1)$, as claimed in Lemma~\ref{lemma.sheet.same}.
\end{example}

\begin{prop}\label{prop.gendim.equal.ssdim} Let $\lambda\vdash n$. Suppose that $\mx =  \semi_\mu+ \nilp \in \fg_\lambda$ is $\lambda$-standard in the sense of Definition~\ref{defn.lambda.compatible}. Let $\semi_\tau \in \fg_\lambda$ denote the semisimple matrix defined in Construction~\ref{constr.semi.tau} with respect to $\mx, \semi_\gm, \nilp$. Then
\[
\dim \left(\Hess(\mx,\hess)\right) = \dim \left( \Hess(\semi_\tau,\hess)\right).
\]
\end{prop}

\begin{proof} 
To prove the equality, we will prove the inequality in both directions. We first show that $\dim \left(\Hess(\semi_\tau,\hess)\right)\leq \dim\left( \Hess(\mx,\hess)\right)$.

Suppose $w\in S_n$ and write $w=yv$ for $y\in W_\tau$ and $v\in {^\tau}W$. Since $\tau$ is a refinement of $\mu$, we may further write $v=xu$ for $x\in W_\mu$ and $u\in {^\mu}W$. Recall that from $u\in {^\mu}W$ we obtain a sequence of Hessenberg functions $\hess_u = (\hess_u^{(1)}, \ldots, \hess_u^{(k)})$ where $k = \len(\mu)$, with each $\hess_u^{(i)}: [\mu_i] \to [\mu_i]$ being a Hessenberg function on $[\mu_i]$.
By Proposition~\ref{prop.ssdim} and Lemma~\ref{lemma.refinement} we have 
\begin{eqnarray}\label{eqn.semidimtau}
\dim  \left(C_w\cap \Hess(\semi_\tau,\hess)\right)  = \ell(y)+ \ell_\hess(v) =  \ell(y)+ \ell_\hess(u) + \ell_{\hess_u}(x). 
\end{eqnarray}
Since $x\in W_\mu \simeq S_{\mu_1}\times S_{\mu_2}\times \cdots S_{\mu_k}$ we may write $x=x_1x_2\ldots x_k$ uniquely,  where for each $i\in[k]$,
\[
x_i \in W_{\Lint{i}(\gm)}:=\left< s_j : j\in \Lint{i}(\gm) \right> \simeq S_{\mu_i}.
\]
Then 
\[
\ell_{\hess_u}(x) = \ell_{\hess_u^{(1)}}(x_1)+ \ell_{\hess_u^{(2)}}(x_2)+\cdots + \ell_{\hess_u^{(k)}}(x_k).
\]
Similarly, since $y\in W_\tau \subseteq W_\mu$ we can write $y = y_1y_2\cdots y_k$ uniquely,  where for each $i\in[k]$, 
\[
y_i\in W_{\Lint{i}(\gm)}:=\left< s_j : j\in \Lint{i}(\gm) \right> \simeq S_{\mu_i}.
\] 
We have $\ell(y) = \ell(y_1)+\ell(y_2)+\cdots + \ell(y_k)$.   Here $y_i$ and $x_j$ commute when $i\neq j$ so:
\[
yx = y_1 y_2 \cdots y_k x_1x_2 \cdots x_k  = (y_1x_1)(y_2x_2)\cdots (y_kx_k)
\]
and
\begin{equation}\label{eqn.ell.plus.ell.hu}
\ell(y)+\ell_{\hess_u}(x) = \sum_{i=1}^k \left( \ell(y_i)+\ell_{\hess_u^{(i)}}(x_i)\right).
\end{equation}

For the next step, recall that the decomposition $v=xu$ is length-additive and $v\in {^\tau}W$. Applying Lemma~\ref{lem.lengthadd} to the length additive decomposition $v^{-1} = u^{-1}x^{-1}$, we get that $x^{-1}(i)>x^{-1}(i+1)$ implies $v^{-1}(i)>v^{-1}(i+1)$, and since $v\in {^\tau}W$, we have $i\in D_\tau$. In particular, we get $x\in {^\tau}W$. This implies that, for each $i$, $y_ix_i$ is a coset decomposition with $y_i \in W_\tau$ and $x_i \in {^\tau}W$. Thus we may apply Proposition~\ref{prop.ssdim} to $y_ix_i$ and obtain 
\begin{align*}
\ell(y_i)+\ell_{\hess_u^{(i)}}(x_i) &= \dim \left(C_{y_ix_i}\cap \Hess(\semi_{\tau^{(i)}}, \hess_u^{(i)}) \right).
\end{align*}
Combining this with~\eqref{eqn.semidimtau} and~\eqref{eqn.ell.plus.ell.hu} yields 
\begin{equation}\label{eqn.dim.Cw.as.sum}
\dim  \left(C_w\cap \Hess(\semi_\tau,\hess)\right)  =  \ell_\hess(u)+ \sum_{i=1}^k   \dim \left( C_{y_ix_i}\cap \Hess(\semi_{\tau^{(i)}}, \hess_u^{(i)}) \right).
\end{equation}

Now suppose $w\in S_n$ such that $\dim  \left(C_w\cap \Hess(\semi_\tau,\hess)\right)$ is maximal as $w$ varies in $S_n$. This implies that, for each $i$, $\dim \left(C_{y_ix_i}\cap \Hess(\semi_{\tau^{(i)}}, \hess_u^{(i)})\right)$ is maximal among all values $\dim \left(C_{z_i} \cap \Hess(\semi_{\tau^{(i)}}, \hess_u^{(i)}\right)$ as $z_i$ varies in $W_{\Lint{i}(\gm)}$. Indeed, if that were not the case, then we could replace some $y_ix_i$ in the factorization of $w$ with some $z_i\in W_{\Lint{i}(\gm)}$ so that 
\[
\dim \left(C_{z_i} \cap \Hess(\semi_{\tau^{(i)}}, \hess_u^{(i)})\right) > \dim \left( C_{y_ix_i}\cap \Hess(\semi_{\tau^{(i)}}, \hess_u^{(i)})\right).
\]
This, in turn, would yield a new permutation $w' = z_1z_2\cdots z_k u \in S_n$ where at least one $z_i$ is not equal to $y_ix_i$, so $w' \neq w$, and $w'$ satisfies  
\[
\dim \left(C_{w'}\cap \Hess(\semi_\tau,\hess)\right) > \dim \left(C_w \cap \Hess(\semi_\tau,\hess)\right),
\]
contradicting the fact that $w$ was selected so that $\dim \left(C_w \cap \Hess(\semi_\tau,\hess)\right)$ is maximal. 

Since $\dim \left(C_{y_ix_i}\cap \Hess(\semi_{\tau^{(i)}}, \hess_u^{(i)})\right)$ is maximal in the above sense for all $i$ and since $\tau^{(i)} = (\lambda^{(i)})^t$, 
Proposition~\ref{prop.steps} implies that there exists $z_i\in W_{\Lint{i}(\gm)}$ such that $C_{z_i} \cap \Hess(\nilp_{\lambda^{(i)}},\hess_u^{(i)})$ is non-empty and
\begin{equation}\label{eqn.induction.step}
\dim \left( C_{y_ix_i}\cap \Hess(\semi_{\tau^{(i)}}, \hess_u^{(i)})  \right)= \dim\left( C_{z_i}\cap \Hess(\nilp_{\lambda^{(i)}}, \hess_u^{(i)}) \right)
\end{equation}
for $1\leq i \leq k$.
Let  $z= z_1z_2\cdots z_k \in W_\gm$. Using the isomorphism of~\eqref{eqn.nil.levi}, we have 
\begin{equation} \label{eqn.nilp.cell.prod}
C_z\cap \Hess(\bar\nilp , \hess_u ) \simeq \prod_{i=1}^k  C_{z_i}\cap \Hess(\nilp_{\lambda^{(i)}} , \hess_u^{(i)} ) 
\end{equation}
and $C_z\cap \Hess(\bar\nilp , \hess_u )\neq \emptyset$ since each $C_{z_i} \cap \Hess(\nilp_{\lambda^{(i)}},\hess_u^{(i)})\neq \emptyset$.

By assumption $C_w \cap \Hess(\semi_\tau,\hess)$ has maximal dimension, so we obtain 
\begin{align*}
\dim \left(\Hess(\semi_\tau,\hess)\right)
&=\dim \left( C_w\cap \Hess(\semi_\tau,\hess) \right) \\
&= \ell_\hess(u) + \sum_{i=1}^k  \dim \left( C_{y_ix_i}\cap \Hess(\semi_{\tau^{(i)}}, \hess_u^{(i)}) \right) \textup{ by~\eqref{eqn.dim.Cw.as.sum}}\\
&= \ell_\hess(u) + \sum_{i=1}^k  \dim \left( C_{z_i} \cap \Hess (\nilp_{\lambda^{(i)}}, \hess_u^{(i)}) \ \right)  \textup{ by~\eqref{eqn.induction.step}}\\
&= \ell_\hess(u) + \dim \left( C_z\cap \Hess(\bar\nilp, \hess_u)\right) \textup{ by~\eqref{eqn.nilp.cell.prod}}\\
&= \dim\left( C_{zu} \cap \Hess( \mx, \hess ) \right) \textup{ by Proposition~\ref{prop.gendim}}\\
&\leq \dim \left(\Hess(\mx,\hess)\right). 
\end{align*}

It remains to prove $\dim \left(\Hess(\mx, \hess)\right) \leq \dim \left( \Hess(\semi_\tau,\hess)\right)$. 
Suppose $w\in S_n$ such that $ C_w\cap \Hess(\mx,\hess)$ is maximum dimensional. Let $w=zu$ be the unique decomposition with $z\in W_\mu$ and $u\in {^\mu}W$. By Proposition~\ref{prop.gendim}, 
\[
\dim \left(C_w\cap \Hess(\mx,\hess)\right) = \ell_\hess(u)+ \dim \left( C_z \cap \Hess(\bar\nilp, h_u) \right)
\]
and $C_z \cap \Hess(\bar\nilp, h_u)\neq \emptyset$.  Let $z=z_1z_2\ldots z_k$ where $z_i\in W_{\Lint{i}(\gm)}$ for each $1\leq i \leq k$. Note that the isomorphism~\eqref{eqn.nilp.cell.prod} holds here, and each $C_{z_i}\cap \Hess(\nilp_{\lambda^{(i)}}, \hess_u^{(i)})$ is nonempty because $C_z \cap \Hess(\bar\nilp, h_u)$ is. Since $\dim \left(C_w\cap \Hess(\mx,\hess)\right)$ is maximal we have 
\begin{equation}\label{eqn.oppinequality1}
\begin{split}
\dim \left(\Hess(\mx,h )\right) &= \dim \left( C_w\cap \Hess(\mx,\hess) \right)  \\
&= \ell_\hess(u)+ \sum_{i=1}^k \dim \left( C_{z_i} \cap \Hess(\nilp_{\lambda^{(i)}}, \hess_u^{(i)}) \right) .
\end{split}
\end{equation}
Since $ \dim \left( C_w\cap \Hess(\mx,\hess) \right)$ is maximal, it follows from an argument similar to that above that $C_{z_i} \cap \Hess(\nilp_{\lambda^{(i)}}, \hess_u^{(i)})$ is a maximum dimensional cell in $\Hess(\nilp_{\lambda^{(i)}}, \hess_u^{(i)})$ for each $i\in[k]$. Each tableau $T_{z_i,\lambda^{(i)}}$ is $\hess_u^{(i)}$-strict, and by Proposition~\ref{prop.steps}, $T_{z_i,\lambda^{(i)}}$ is also column-increasing. Recall that $\tau^{(i)}=(\lambda^{(i)})^t$. Proposition~\ref{prop.steps}(1) also tells us that, for each $i$, 
\begin{equation}\label{eqn.Czi}
\dim \left( C_{z_i} \cap \Hess(\nilp_{\lambda^{(i)}}, \hess_u^{(i)}) \right) = \dim \left( C_{z_i} \cap \Hess(\semi_{\tau^{(i)}}, \hess_u^{(i)}) \right).
\end{equation}
For each $i$, write $z_i = y_i x_i$ for a unique pair $y_i\in W_{\tau^{(i)}}$ and $x_i \in {^{\tau^{(i)}}}W$. Then by Proposition~\ref{prop.ssdim}, 
\begin{equation}\label{eqn.last.step}
\dim \left( C_{z_i} \cap \Hess(\semi_{\tau^{(i)}}, \hess_u^{(i)}) \right) = \ell(y_i)+ \ell_{\hess_u^{(i)}} (x_i).
\end{equation}
Let $y = y_1 y_2 \cdots y_k$ and $x = x_1x_2 \ldots x_k$. 
Now by~\eqref{eqn.last.step} and~\eqref{eqn.Czi}, the equality ~\eqref{eqn.oppinequality1} becomes 
\begin{align*}
\dim \left(\Hess(\mx,h )\right) &= \ell_\hess(u)+ \sum_{i=1}^k \left(  \ell(y_i)+ \ell_{\hess_u} (x_i) \right) \\
& = \ell_\hess(u)+ \ell(y)+\ell_{\hess_u}(x) \textup{ by~\eqref{eqn.ell.plus.ell.hu}} \\
& = \ell(y) + \ell_\hess(xu) \textup{ by Lemma~\ref{lemma.refinement}.}  
\end{align*}
Note that we can apply Lemma~\ref{lemma.refinement} in the last equality above because $x \in W_\gm$ and $u \in {^\gm}W$.
Then by Proposition~\ref{prop.ssdim},
\[
\dim \left(\Hess(\mx,h )\right) = \dim \left( C_ w \cap \Hess(\semi_\tau,\hess)\right) \leq \dim \left(\Hess(\semi_\tau,\hess)\right).
\]
Thus, we may now conclude $\dim \left(\Hess(\mx,h )\right)  = \dim \left(\Hess(\semi_\tau,\hess)\right)$.
\end{proof}

\begin{proof}[Proof of Theorem~\ref{thm.main.general}] Let $\mx\in \fg_\lambda$. Replacing $\mx$ by a conjugate if need be, we may assume that $\mx$ is $\lambda$-standard; let $\semi_\tau$ be the corresponding diagonal matrix from Construction~\ref{constr.semi.tau}. By Theorem~\ref{thm.ss-nilp}, Lemma~\ref{lemma.sheet.same}, and Proposition~\ref{prop.gendim.equal.ssdim},
\[
\dim \left(\Hess(\mx, \hess)\right) = \dim \left(\Hess(\semi_\tau ,\hess)\right) = \dim \left(\Hess(\nilp_\lambda, \hess)\right).
\]
In particular, every Hessenberg variety defined by $\mx\in \fg_\lambda$ has dimension equal to that of a nilpotent Hessenberg variety corresponding to the unique nilpotent orbit in $\fg_\lambda$. Thus, they must all have the same dimension.
\end{proof}


\section{Open questions}\label{sec.open}

We state some open questions. Recall the set 
\[
\cm_{\lambda, \hess}:= \{w\in S_n : \tabExtra{w,\lambda} \text{ is $\hess$-strict and column increasing}\}.
\]
Suppose the Hessenberg--Schubert cell $C_w\cap \Hess(\nilp_\lambda, \hess)$ in the nilpotent Hessenberg variety $\Hess(\nilp_\lambda, \hess)$ is maximum dimensional. It follows that its closure is a maximum dimensional irreducible component of $\Hess(\nilp_\lambda, \hess)$, and by Proposition~\ref{prop.steps}(2), that $w\in \cm_{\lambda, \hess}$. The same proposition implies $C_w\cap \Hess(\semi_{\lambda^t},\hess)$ is also a maximum dimensional cell, and its closure is therefore a maximum dimensional irreducible component of $\Hess(\semi_{\lambda^t},\hess)$. However, unlike the nilpotent case, the semisimple Hessenberg variety $\Hess(\semi_{\lambda^t}, \hess)$ may also have other maximum dimensional components equal to closures of Hessenberg--Schubert cells indexed by permutations \emph{not} in $\cm_{\lambda, \hess}$. We conclude that $\Hess(\semi_{\lambda^t}, \hess)$ has at least as many maximum dimensional irreducible components as $\Hess(\nilp_\lambda, \hess)$. 

The Poincar\'e polynomial of a variety $\cz$ records the graded dimension of its cohomology ring:
\[
\Poin(\cz;t) := \sum_{i= 0}^{d} \dim \left( H^i(\cz)  \right)t^i, 
\]
where $d=\dim \cz$.
The leading coefficient is precisely the number of maximum dimensional irreducible components of $\cz$. In light of the discussion in the previous paragraph, we make the following conjecture, which has been confirmed to hold with $\mx = \semi_{\lambda^t}$ for all $\lambda\vdash n$ with $n\leq 7$. 

\begin{conjecture} For all $\lambda \vdash n$ and $\mx\in \fg_\lambda$, the Poincar\'e polynomial of $\Hess(\mx, \hess)$ dominates that of $\Hess(\nilp_\lambda, \hess)$. That is, if
\[
\Poin(\Hess(\mx, \hess); t) = \sum_{i\geq 0} a_i t^i \; \text{ and }\; \Poin(\Hess(\nilp_\lambda, \hess); t) = \sum_{i\geq 0 } b_i t^i,
\]
then we have $a_i\geq b_i$ for all $i$.
\end{conjecture}

When $\mx$ is a regular matrix, that is, when $\mx\in \fg_{(n)}$, the conjecture above follows immediately from~\cite[Theorem 35]{Brosnan-Chow}. In this case, there is also a closed formula for the dimension of the corresponding Hessenberg variety, namely
\[
\dim (\Hess(\mx, \hess)) = \sum_{i=1}^n (\hess(i)-i).
\]
This is the ``expected dimension'' for type $A$ Hessenberg varieties (defined using any matrix), in the sense that any type $A$ Hessenberg variety is locally determined by at most $\sum_{i=1}^n (\hess(i)-i)$ equations (see, for example, \cite[Section 4]{Insko-Tymoczko-Woo}). Applying~\cite[Exercise 3.22]{Hartshorne}, any Hessenberg variety of the expected dimension is pure dimensional. This motivates the following:

\begin{problem} Fix $\mx\in \fg_\lambda$ with $\lambda\neq (n)$. Classify the set of all Hessenberg functions $\hess$ such that $\dim \left(\Hess(\mx, \hess) \right)= \sum_{i=1}^n(\hess(i)-i)$.
\end{problem}

If $\Hess(\mx, \hess)$ has the expected dimension, then each of its irreducible components is the closure of a Hessenberg--Schubert cell of maximum dimension. In this case, the affine paving, and methods applied in the paper, can be used to determine all irreducible components of $\Hess(\mx, \hess)$ by identifying these top-dimensional cells. It would be interesting to classify these components explicitly in the ``expected dimension setting'' for $\Hess(\nilp_\lambda, \hess)$ by identifying the corresponding subset of $\cm_{\lambda, h}$. We leave this classification as a subject for future work.

\bibliographystyle{alpha}
\bibliography{ref.bib}

@article {Insko-Tymoczko-Woo,
    AUTHOR = {Insko, Erik and Tymoczko, Julianna and Woo, Alexander},
     TITLE = {A formula for the cohomology and {$K$}-class of a regular
              {H}essenberg variety},
   JOURNAL = {J. Pure Appl. Algebra},
  FJOURNAL = {Journal of Pure and Applied Algebra},
    VOLUME = {224},
      YEAR = {2020},
    NUMBER = {5},
     PAGES = {106230, 14},
      ISSN = {0022-4049,1873-1376},
   MRCLASS = {14M12 (05E99 14M15 14N15)},
  MRNUMBER = {4046236},
MRREVIEWER = {Ryan\ David\ Kinser},
       DOI = {10.1016/j.jpaa.2019.106230},
       URL = {https://doi-org.libproxy.washu.edu/10.1016/j.jpaa.2019.106230},
}

@article {Brosnan-Chow,
    AUTHOR = {Brosnan, Patrick and Chow, Timothy Y.},
     TITLE = {Unit interval orders and the dot action on the cohomology of
              regular semisimple {H}essenberg varieties},
   JOURNAL = {Adv. Math.},
  FJOURNAL = {Advances in Mathematics},
    VOLUME = {329},
      YEAR = {2018},
     PAGES = {955--1001},
      ISSN = {0001-8708,1090-2082},
   MRCLASS = {05E05 (14M15)},
  MRNUMBER = {3783432},
MRREVIEWER = {Marko\ Radovanovi\'c},
       DOI = {10.1016/j.aim.2018.02.020},
       URL = {https://doi-org.libproxy.washu.edu/10.1016/j.aim.2018.02.020},
}

@book {Hartshorne,
    AUTHOR = {Hartshorne, Robin},
     TITLE = {Algebraic geometry},
    SERIES = {Graduate Texts in Mathematics},
    VOLUME = {No. 52},
 PUBLISHER = {Springer-Verlag, New York-Heidelberg},
      YEAR = {1977},
     PAGES = {xvi+496},
      ISBN = {0-387-90244-9},
   MRCLASS = {14-01},
  MRNUMBER = {463157},
MRREVIEWER = {Robert\ Speiser},
}

@article {ADGH,
    AUTHOR = {Abe, Hiraku and DeDieu, Lauren and Galetto, Federico and
              Harada, Megumi},
     TITLE = {Geometry of {H}essenberg varieties with applications to
              {N}ewton-{O}kounkov bodies},
   JOURNAL = {Selecta Math. (N.S.)},
  FJOURNAL = {Selecta Mathematica. New Series},
    VOLUME = {24},
      YEAR = {2018},
    NUMBER = {3},
     PAGES = {2129--2163},
      ISSN = {1022-1824,1420-9020},
   MRCLASS = {14M17 (14M10 14M25)},
  MRNUMBER = {3816501},
MRREVIEWER = {Lucas\ Fresse},
       DOI = {10.1007/s00029-018-0405-3},
       URL = {https://doi.org/10.1007/s00029-018-0405-3},
}

@article {Goldin-Precup,
    AUTHOR = {Goldin, Rebecca and Precup, Martha},
     TITLE = {A flat family of matrix {H}essenberg schemes over the minimal
              sheet},
   JOURNAL = {J. Algebra},
  FJOURNAL = {Journal of Algebra},
    VOLUME = {692},
      YEAR = {2026},
     PAGES = {123--172},
      ISSN = {0021-8693,1090-266X},
   MRCLASS = {14M15 (05E14 13P10)},
  MRNUMBER = {5002214},
       DOI = {10.1016/j.jalgebra.2025.11.026},
       URL = {https://doi.org/10.1016/j.jalgebra.2025.11.026},
}

@incollection {Lusztig,
    AUTHOR = {Lusztig, George},
     TITLE = {On conjugacy classes in a reductive group},
 BOOKTITLE = {Representations of reductive groups},
    SERIES = {Progr. Math.},
    VOLUME = {312},
     PAGES = {333--363},
 PUBLISHER = {Birkh\"auser/Springer, Cham},
      YEAR = {2015},
      ISBN = {978-3-319-23442-7; 978-3-319-23443-4},
   MRCLASS = {20G15 (20G07)},
  MRNUMBER = {3495802},
MRREVIEWER = {William\ M.\ McGovern},
       DOI = {10.1007/978-3-319-23443-4\_12},
       URL = {https://doi.org/10.1007/978-3-319-23443-4_12},
}

@article {Tymoczko2006,
    AUTHOR = {Tymoczko, Julianna S.},
     TITLE = {Linear conditions imposed on flag varieties},
   JOURNAL = {Amer. J. Math.},
  FJOURNAL = {American Journal of Mathematics},
    VOLUME = {128},
      YEAR = {2006},
    NUMBER = {6},
     PAGES = {1587--1604},
      ISSN = {0002-9327,1080-6377},
   MRCLASS = {14M15 (05E15 14L35 20G20)},
  MRNUMBER = {2275912},
MRREVIEWER = {James\ E.\ Humphreys},
       URL =
              {http://muse.jhu.edu/journals/american_journal_of_mathematics/v128/128.6tymoczko.pdf},
}

@article {Precup2013,
    AUTHOR = {Precup, Martha},
     TITLE = {Affine pavings of {H}essenberg varieties for semisimple
              groups},
   JOURNAL = {Selecta Math. (N.S.)},
  FJOURNAL = {Selecta Mathematica. New Series},
    VOLUME = {19},
      YEAR = {2013},
    NUMBER = {4},
     PAGES = {903--922},
      ISSN = {1022-1824,1420-9020},
   MRCLASS = {14L35 (14F25 14M15 17B08)},
  MRNUMBER = {3131491},
MRREVIEWER = {Nicolas\ Perrin},
       DOI = {10.1007/s00029-012-0109-z},
       URL = {https://doi.org/10.1007/s00029-012-0109-z},
}

@book {BB05,
    AUTHOR = {Bj\"orner, Anders and Brenti, Francesco},
     TITLE = {Combinatorics of {C}oxeter groups},
    SERIES = {Graduate Texts in Mathematics},
    VOLUME = {231},
 PUBLISHER = {Springer, New York},
      YEAR = {2005},
     PAGES = {xiv+363},
      ISBN = {978-3540-442387; 3-540-44238-3},
   MRCLASS = {05-01 (05E15 20F55)},
  MRNUMBER = {2133266},
MRREVIEWER = {Jian-yi\ Shi},
}

@article {Ji-Precup2019,
    AUTHOR = {Ji, Caleb and Precup, Martha},
     TITLE = {Hessenberg varieties associated to ad-nilpotent ideals},
   JOURNAL = {Comm. Algebra},
  FJOURNAL = {Communications in Algebra},
    VOLUME = {50},
      YEAR = {2022},
    NUMBER = {4},
     PAGES = {1728--1749},
      ISSN = {0092-7872,1532-4125},
   MRCLASS = {14M15 (05A05 14L35)},
  MRNUMBER = {4391520},
MRREVIEWER = {Praise\ Adeyemo},
       DOI = {10.1080/00927872.2021.1988629},
       URL = {https://doi.org/10.1080/00927872.2021.1988629},
}

 \end{document}